\documentclass[11pt,reqno]{amsart}
\usepackage[margin=1.1in]{geometry}
\usepackage{amsmath,amsthm,amssymb}
\usepackage{mathrsfs}

\date{17 January, 2022}

\usepackage{enumitem}
\usepackage{pgf,tikz,pgfplots}
\usetikzlibrary{arrows}
\usepackage[unicode]{hyperref}
\usepackage[caption=false]{subfig}
\hypersetup{colorlinks=true, linkcolor=blue, citecolor=blue, filecolor=blue, urlcolor=blue}
\usepackage{comment}
\usepackage{tcolorbox} 
\usepackage{chngcntr}
\usepackage{mathabx}
\usepackage{bm}
\usepackage{etoolbox}
\numberwithin{equation}{section} 

\usepackage[comma,sort&compress, numbers]{natbib}

\newcommand{\one}{{\bf 1}}

\newtheorem{theorem}{Theorem}[section]

\newtheorem{prop}[theorem]{Proposition}
\newtheorem{problem}[theorem]{Problem}

\newtheorem{lemma}[theorem]{Lemma}

\newtheorem{claim}{Claim}[section] 

\theoremstyle{definition} 
\newtheorem{definition}[theorem]{Definition}
\newtheorem{remark}[theorem]{Remark}
\newtheorem{example}[theorem]{Example}

\allowdisplaybreaks

\begingroup
  \divide\count255 by 60
  \multiply\count255 by -60
  \advance\count255 by \time
  \ifnum \count255 < 10 \xdef\klockan{\the\count1.0\the\count255}
  \else\xdef\klockan{\the\count1.\the\count255}\fi
\endgroup

\newenvironment{alphenumerate}[1][0pt]{
\addtolength{\leftmargini}{#1}\begin{enumerate}
 }{\end{enumerate}}

\newcommand\set[1]{\ensuremath{\{#1\}}}
\newcommand\bigset[1]{\ensuremath{\bigl\{#1\bigr\}}}

\newcommand\bigpar[1]{\bigl(#1\bigr)}
\newcommand\Bigpar[1]{\Bigl(#1\Bigr)}

\newcommand\lrpar[1]{\left(#1\right)}
\newcommand\sqpar[1]{[#1]}
\newcommand\bigsqpar[1]{\bigl[#1\bigr]}

\newcommand\Aut{\mathrm{Aut}}
\newcommand\Spec{\mathrm{Spec}}
\newcommand\cF{\mathcal{F}}
\newcommand\cS{\mathcal{S}}
\newcommand\cW{\mathcal{W}}
\newcommand\Sref[1]{\cS_{\eqref{#1}}}
\newcommand\bSref[1]{\overline\cS_{\eqref{#1}}}
\newcommand\Srefdummy{\cS_{(\cdot)}}
\newcommand\bSrefdummy{\overline\cS_{(\cdot)}}
\newcommand\oi{[0,1]}
\newcommand\qq{^{1/2}}

\newcommand\norm[1]{\|#1\|_2}
\newcommand\innprod[1]{\langle#1\rangle}
\newcommand\gd{\delta}
\newcommand\dd{\,\mathrm{d}}
\newcommand\ddx{\dd x}
\newcommand\ddy{\dd y}
\newcommand\ga{\alpha}
\newcommand\gl{\lambda}
\newcommand\gL{\Lambda}
\newcommand\gs{\sigma}
\newcommand\gss{\sigma^2}
\newcommand\tM{\hat M}
\newcommand\dto{\overset{D}{\rightarrow}}
\newcommand\dtoo{\overset{D}{\longrightarrow}}
\newcommand\E{\mathbb{E}}
\newcommand\Var{\operatorname{Var}}
\newcommand\Ka{{K_{\set a}}}
\newcommand\Kb{{K_{\set b}}}
\newcommand\Kab{{E_{\set{a,b}}}}

\newcommand\Ki{{K_{\set{1}}}}
\newcommand\Qab{{K_{\{a,b\}}}}
\newcommand\Qij{{K_{\{1,2\}}}}
\newcommand\tij{t_{1,2}}
\newcommand\fett{f_{(1)}}
\newcommand\ftva{f_{(2)}}
\newcommand\setab{\set{a,b}}
\newcommand\Wi{\cW_1}

\allowdisplaybreaks

\begin{document}

\title[Fluctuations of Subgraph Counts in Graphon Based Random Graphs]{Fluctuations of Subgraph Counts in Graphon Based Random Graphs}
\author[Bhattacharya, Chatterjee, Janson]{Bhaswar B. Bhattacharya \and
  Anirban Chatterjee \and Svante Janson}
\thanks{
BBB partly supported by NSF CAREER Grant DMS-2046393 and a Sloan research fellowship. SJ partly supported by the Knut and Alice Wallenberg Foundation.}
\address{Department of Statistics\\ University of Pennsylvania\\ Philadelphia\\ PA 19104\\ United States}
\email{bhaswar@wharton.upenn.edu}
\address{Department of Statistics\\ University of Pennsylvania\\ Philadelphia\\ PA 19104\\ United States}
\email{anirbanc@wharton.upenn.edu}
\address{Department of Mathematics, Uppsala University, PO Box 480, SE-751 06 Uppsala, Sweden}
\email{svante.janson@math.uu.se}
\keywords{Inhomogeneous random graphs, generalized $U$-statistics, graphons, limit theorems, subgraph counts.}
\subjclass[2010]{05C80, 60F05, 05C60}

\begin{abstract} Given a graphon $W$ and a finite simple graph $H$, with vertex set $V(H)$, denote by $X_n(H, W)$ the number of copies of $H$ in a $W$-random graph on $n$ vertices. The asymptotic distribution of $X_n(H, W)$ was recently obtained by  Hladk\'y, Pelekis, and \v{S}ileikis \cite{hladky2019limit} in the case where $H$ is a clique. In this paper, we extend this result to any fixed graph $H$. Towards this we introduce a notion of $H$-regularity of graphons and show that if the graphon $W$ is not $H$-regular, then $X_n(H, W)$ has Gaussian fluctuations with scaling $n^{|V(H)|-\frac{1}{2}}$. On the other hand, if $W$ is $H$-regular, then the fluctuations are of order  $n^{|V(H)|-1}$ and the limiting distribution of  $X_n(H, W)$ can have both Gaussian and non-Gaussian components, where the non-Gaussian component is a (possibly) infinite weighted sum of centered chi-squared random variables with the weights determined by the spectral properties of a graphon derived from $W$. Our proofs use the asymptotic theory of generalized $U$-statistics developed by Janson and Nowicki \cite{janson1991asymptotic}. We also investigate the structure of $H$-regular graphons for which either the Gaussian or the non-Gaussian component of the limiting distribution (but not both)  is degenerate. Interestingly, there are also $H$-regular graphons $W$ for which both the Gaussian or the non-Gaussian components are degenerate, that is, $X_n(H, W)$ has a degenerate limit even under the scaling $n^{|V(H)|-1}$. We give an example of this degeneracy with $H=K_{1, 3}$ (the 3-star) and also establish non-degeneracy in a few examples. This naturally leads to interesting open questions on higher-order degeneracies.
\end{abstract}

\maketitle


\section{Introduction} 

A \textit{graphon} is a measurable function $W: [0,1]^2\rightarrow[0,1]$
which is symmetric, that is, $W(x,y)=W(y,x)$, for all
$x,y\in[0,1]$. Graphons arise as the limit objects of sequences of large
graphs and has received phenomenal attention over the last few years. 
They 
provide a bridge between combinatorics and analysis, and have found
applications in  several disciplines including statistical physics,
probability, and statistics; see for example
\cite{borgs2008convergent,borgs2012convergent,bmpotts,chatterjee2013estimating,chatterjee2011large}.
 For a detailed exposition of the theory of graph limits, we refer to Lov\' asz \cite{lovasz2012large}.  Graphons provide a natural sampling procedure for generating  inhomogeneous variants of the classical Erd\H{o}s--R\'enyi random graph, a concept that has been proposed independently by various authors (see \cite{bollobas2007phase,diaconis1981statistics,lovasz2006limits,boguna2003class} among others). Formally, given a graphon $W:[0,1]^{2}\rightarrow[0,1]$, a $W$-{\it random graph} on the set of vertices $[n]:=\{1,2, \ldots, n\}$, hereafter denoted by $G(n, W)$, is obtained by connecting the vertices $i$ and $j$ with probability $W(U_{i},U_{j})$ independently for all $1 \leq i < j \leq n$, where $\{U_{i}: 1 \leq i \leq n\}$ is an i.i.d.\ sequence of $U[0,1]$ random variables. An alternative way to achieve this sampling is to generate i.i.d.\ sequences $\{U_{i}: 1 \leq i \leq n\}$ and $\{Y_{ij}: 1\leq i<j \leq n\}$ of $U[0,1]$ random variables and then assigning the edge $(i, j)$ whenever $\{Y_{ij}\leq W(U_{i},U_{j})\}$, for $1 \leq i < j \leq n$. Observe that setting $W = W_p \equiv p\in [0,1]$ gives the classical (homogeneous) Erd\H{o}s--R\'enyi random graph model, where every edge is present independently with constant probability $p$.

Counts of subgraphs encode important structural information about the
geometry of a  network. In fact, the convergence of a sequence of
finite graphs to a graphon is precisely determined by the convergence of its
subgraph densities. As a consequence, understanding the asymptotic
properties of subgraph counts in $W$-random graphs is a problem of central
importance in graph limit theory. To this end, given a finite
graph $H=(V(H), E(H))$ denote by $X_{n}(H,W)$
the number of copies of $H$ in the $W$-random graph $G(n,W)$. More formally,  
\begin{align}\label{eq:XHW}
X_{n}(H,W)=\sum_{1\leq i_{1}<\cdots<i_{|V(H)|}\leq n}\sum_{H'\in
  \mathscr{G}_H(\{i_{1},\ldots, i_{|V(H)|} \}) } 
\prod_{(i_{s}, i_{t}) \in  E(H')}
\one\left\{Y_{i_{a}i_{b}}\leq W(U_{i_{a}},U_{i_{b}})\right\}, 
\end{align}
where, for any set $S \subseteq [n]$, $\mathscr G_H(S)$ denotes the
collection of all subgraphs of the complete graph $K_{|S|}$ on the vertex
set $S$ which are isomorphic to $H$. 
(We count unlabelled copies of $H$. Several other authors count labelled
copies, which multiplies $X_n(H,W)$ by $|\Aut(H)|$, cf.~\eqref{eq:fact_Aut}.)
The asymptotic distribution of $X_n(H,W_p)$ in the Erd\H{o}s--R\'enyi
model, where $W = W_p  \equiv p$,
has been classically studied
(in general with $p=p(n)$)
using various tools such as $U$-statistics
\cite{Nowicki1989,nowicki1988subgraph}, 
method of moments \cite{rucinski1988small}, 
Stein's method \cite{barbour1989central}, 
and martingales 
\cite{SJ79, SJ94},
see also \cite[Chapter 6]{JLR}, 
and the precise conditions under
which $X_n(H, W_p)$ is asymptotically normal are well-understood
\cite{rucinski1988small}. 
In particular, when $p \in (0, 1)$ is fixed, $X_n(H, W_p)$ is asymptotically
normal for any finite graph $H$ that is non-empty, i.e., has at least one edge.

In this paper we study the asymptotic distribution of $X_n(H, W)$ for
general graphons $W$.
This problem has received significant attention recently, beginning with the
work of F{\'e}ray, M{\'e}liot, and Nikeghbali~\cite{feray2020graphons},
where the asymptotic normality for homomorphism densities in general
$W$-random graphs was derived using the framework of mod-Gaussian
convergence. 
Using this machinery the authors also obtained moderate deviation
principles and local limit theorems for the homomorphism densities in this
regime. Very recently, using Stein's method, rates of convergence to
normality (Berry--Esseen type bounds) 
have been derived as well, see
\cite{kaur2021higher} (which also contain further related results) and
\cite{zhang2021berryesseen}. 
See also \cite{delmas2021asymptotic} and the references therein 
for further 
results.

However,
interestingly, the limiting normal distribution of the subgraph counts
obtained in 
\cite{feray2020graphons} can be degenerate depending on the structure of the
graphon $W$. This phenomenon was 
observed in \cite{feray2020graphons}, and it was
explored in detail in the  recent
paper of Hladk\'y, Pelekis, and \v{S}ileikis~\cite{hladky2019limit} for the
case where $H=K_r$ is the $r$-clique, for some $r \geq 2$. They showed that
the usual Gaussian limit is degenerate when a certain regularity function,
which encodes the homomorphism density of $K_r$ incident on a given `vertex'
of $W$, is constant almost everywhere (a.e.). In this case, the graphon $W$
is said to be $K_r$-regular and the asymptotic distribution of $X_n(K_r, W)$
(with another normalization, differing by a factor $n\qq$) 
has both Gaussian and non-Gaussian components. 
In the present paper we extend this
result to any fixed graph $H$. To this end, we introduce the analogous
notion of $H$-regularity and show that the fluctuations of $X_n(H, W)$
depends on whether or not $W$ is $H$-regular. In particular, if $W$ is not
$H$-regular, then, 
$X_n(H, W)$ is asymptotically Gaussian,
using a normalization factor $n^{|V(H)|-1/2}$.
However, if $W$ is
$H$-regular, 
then the normalization factor becomes $n^{|V(H)|-1}$ and yields
a limiting distribution of  $X_n(H, W)$ that has, in general, 
a Gaussian
component and another independent (non-Gaussian) component which is a
(possibly) infinite weighted sum of centered chi-squared random
variables. Here, the weights are determined by the spectrum of a graphon
obtained from the 2-point conditional densities of $H$ in $W$, that is, the
density of $H$ in $W$ when two vertices of $H$ are mapped to two `vertices'
of $W$, averaged over all pairs of vertices of $H$. The results are formally stated in Theorem \ref{thm:asymp_count}. Unlike in
\cite{hladky2019limit} which uses the method of moments, our proofs employ
the orthogonal decomposition for generalized $U$-statistics developed by
Janson and Nowicki \cite{janson1991asymptotic}
(see also \cite[Chapter 11.3]{SJIII}). 
This avoids cumbersome
moment calculations and provides a more streamlined framework for dealing
with the asymmetries of general subgraphs.

There are also exceptional cases, where $W$ is $H$-regular and
normalization of $X_n(H,W)$ by $n^{|V(H)|-1}$ also yields a degenerate limit;
then a non-trivial limit can be found by another normalization.
(We ignore trivial cases when $X_n(H,W)$ is deterministic.)
This cannot happen when $H=K_r$ 
as shown in \cite{hladky2019limit}, but we give an example of this degeneracy with $H=K_{1,3}$ (the 3-star); see Example \ref{example:3star}. We also show that this higher-order degeneracy cannot happen for $H=C_4$ (the 4-cycle) and $H=K_{1,2}$ (the 2-star); see Theorem \ref{thm:condition} and Theorem \ref{thm:K12Z}, respectively. It is an open problem to decide for which graphs $H$ such higher-order degeneracies may occur.

We also study the structure of $W$ is when it is $H$-regular
and one (but not both) of the two components of the limit distribution
in Theorem \ref{thm:asymp_count}(2) vanishes, so that the limit distribution either is normal, or lacks a normal component. In particular, we show that if $H$ is bipartite and $W$ is $H$-regular, then the limit  lacks a normal component if and only if $W$ is $\{0, 1\}$-valued almost everywhere (Theorem \ref{thm:bizero}).

\subsection{Organization} The rest of the paper is organized as follows. The limit theorems for the subgraph counts are presented in  Section \ref{S2}. We compute the limits in some examples in Section \ref{sec:example}. Degeneracies of the asymptotic distributions are discussed in Section \ref{Sdeg}. The main results are proved in Sections \ref{sec:proof}--\ref{sec:K12Zpf}.

\section{Asymptotic Distribution of Subgraph Counts in $W$-Random Graphs}
\label{S2}

In this section we will state our main result on the asymptotic distribution
$X_n(H, W)$. The section is organized as follows: In Section
\ref{sec:preliminaries} we recall some basic definitions about graphons. The
notions of conditional homomorphism density and $H$-regularity are
introduced in Section  \ref{sec:conditional_densities}. 
Some 
spectral properties of the integral operator corresponding to a graphon are described in Section \ref{sec:aux}.
The
result is formally stated in Section \ref{sec:XHW}.

\subsection{Preliminaries}
\label{sec:preliminaries}

 A quantity that will play a central role in our analysis the {\it homomorphism density} of a fixed multigraph $F=(V(F), E(F))$ (without loops) in a graphon $W$, which is defined as: 
\begin{align}\label{eq:tFW}
t(F,W)=\int_{[0,1]^{|V(F)|}}\prod_{(s, t) \in E(F)}W(x_{a},x_{b})\prod_{a=1}^{|V(F)|}\dd x_{a}. 
\end{align}
Note that this is the natural continuum analogue of the homomorphism density of a  fixed graph $F=(V(F), E(F))$ into  finite (unweighted) graph $G=(V(G), E(G))$ which is defined as: 
\begin{align}
t(F, G) :=\frac{|\hom(F,G)|}{|V (G)|^{|V (F)|}},
\end{align}
where  $|\hom(F,G)|$ denotes the number of homomorphisms of $F$ into $G$. 
In fact, it is easy to verify that $t(F, G) = t(F, W^{G})$, where $W^{G}$ is the {\it empirical graphon} associated with the graph $G$ which defined as: 
\begin{align}\label{eq:emp_graph}
	W^G(x, y) =\boldsymbol 1\{(\lceil |V(G)|x \rceil, \lceil |V(G)|y \rceil)\in E(G)\}.
\end{align} 
(In other words, to obtain the empirical graphon $W^G$ from the graph $G$, partition $[0, 1]^2$ into $|V(G)|^2$ squares of side length $1/|V(G)|$, and let $W^G(x, y)=1$ in the $(i, j)$-th square if $(i, j)\in E(G)$, and 0 otherwise.) 

Let $H=(V(H), E(H))$ be a simple graph.
For convenience, we will throughout the paper assume that $V(H)=\{1, 2, \ldots,|V(H)|\}$.
Then,
the homomorphism density defined \eqref{eq:tFW} can also
interpreted as the probability that a $W$-random graph on $|V(H)|$ vertices 
contains
$H$, that is,  
\begin{align}
t(H,W)=\mathbb{P}(G(|V(H)|,W)\supseteq H).
\end{align} 
To see this, recall the construction of a $W$-random graph and note from \eqref{eq:tFW} that, 
\begin{align}
t(H,W)=\mathbb{E}\left[\prod_{(a,b)\in E(H)}W(U_{a},U_{b})\right] & = \mathbb{E}\left[\prod_{(a,b)\in E(H)}\one\{ Y_{ab}\leq W(U_{a},U_{b}) \}\right] \nonumber \\ 
& = \mathbb{E}\left[\one\{G(|V(H)|,W)\supseteq H \}\right]. 
\end{align}

Next, recalling \eqref{eq:XHW} note that 
\begin{align}\label{eq:XHW_mean}
    \mathbb{E}X_{n}(H,W)
    &=\sum_{1\leq i_{1}<\cdots<i_{|V(H)|}\leq n}\sum_{H'\in\mathscr{G}_H(\{i_{1},\ldots, i_{|V(H)|}\})}t(H,W) \nonumber \\
    &={n\choose |V(H)|} \left |\mathscr{G}_H( \{ 1,\ldots,  |V(H)| \} ) \right| \cdot t(H,W)
\end{align}
where the  last equality follows since the number of subgraphs of
$K_{|V(H)|}$ on $\{i_{1},\ldots, i_{|V(H)|}\}$ isomorphic to $H$ is the same for any collection of distinct indices $1\leq i_{1}<\cdots<i_{|V(H)|}\leq n$. Clearly, 
\begin{align}\label{eq:fact_Aut} 
\left |\mathscr{G}_H( \{ 1,\ldots,  |V(H)| \} ) \right| =    \frac{|V(H)|!}{|\Aut(H)|}, 
\end{align}
where $\Aut(H)$ is the collection of all automorphisms of $H$, that is, the collection of permutations $\sigma$ of the vertex set $V(H)$ such that $(x, y) \in E(H)$ if and only if $(\sigma(x), \sigma(y)) \in E(H)$. This implies, from \eqref{eq:XHW_mean}, 
\begin{align}\label{EXHW} 
\mathbb{E}X_{n}(H,W) = \frac{(n)_{|V(H)|}}{|\Aut(H)|}t(H,W) , 
\end{align} 
where $(n)_{|V(H)|}:=n(n-1)\cdots(n-|V(H)|+1)$.

\subsection{Conditional Homomorphism Densities and $H$-Regularity}
\label{sec:conditional_densities}

In this section we will formalize the notion of $H$-regularity of a graphon
$W$. To this end, we need to introduce the notion of conditional
homomorphism densities. Throughout, we will assume  $H=(V(H), E(H))$ is a
non-empty simple  graph with vertices labeled $V(H)=\{1, 2, \ldots,
|V(H)|\}$.

\begin{definition}\label{defn:marked_density} 
Fix $1 \leq K \leq |V(H)|$ and an ordered set $\bm{a}= (a_1, a_2, \ldots, a_K)$ of distinct vertices $a_1, a_2, \ldots, a_K \in V(H)$. Then the $K$-{\it point conditional homomorphism density function} of $H$ in $W$ given $\bm{a}$ is defined as: 
\begin{align}\label{ta}
    t_{\bm{a}}(\bm{x},H,W) &:=\mathbb{E}\left[\prod_{(a,b)\in E(H)}W(U_{a},U_{b})\Bigm|U_{a_{j}}=x_{j}, \text{ for } 1\leq j \leq K \right] \notag\\ 
 &\phantom:=
\mathbb{P}\left(G(|V(H)|,W)\supseteq H \bigm| U_{a_{j}}=x_{j}, \text{ for } 1\leq j \leq K \right), 
\end{align}
where $\bm{x}= (x_1, x_2, \ldots, x_K)$. In other words,  $t_{\bm{a}}(\bm{x},H,W)$ is the homomorphism density of $H$ in the graphon $W$ when the vertex $a_j \in V(H)$ is marked with the value $x_j \in [0, 1]$, for $1 \leq j \leq K$. 
\end{definition}

The conditional homomorphism densities will play a crucial role in the description of the limiting distribution of $X_n(H, W)$. In particular, the $H$-regularity of a graphon $W$ is determined by the 1-point conditional homomorphism densities, which we formalize below: 

\begin{definition}[$H$-regularity of a graphon]\label{defn:regular} 
A graphon $W$ is said to be $H$-\emph{regular} if
\begin{align}\label{eq:H_regular}
 \overline{t}(x,H,W) :=   \frac{1}{|V(H)|}\sum_{a=1}^{|V(H)|}t_a(x,H,W)=t(H,W),
\end{align}
for almost every $x \in [0, 1]$. 
\end{definition} 

Note that in \eqref{eq:H_regular} it is enough to assume that
$\overline{t}(x,H,W) $ is a constant for almost every $x \in [0, 1]$. This
is because 
\begin{align}\label{intta}
\int_0^1 t_{a}(x,H,W) \dd x = t(H, W),  
\end{align}
for all $a \in V(H)$. Hence, if 
$\overline{t}(x,H,W) $ is a constant a.e., then the constant must be $t(H, W)$. Therefore, in other words, a graphon $W$ is $H$-regular if the homomorphism density of $H$ in $W$ when one of the vertices of $H$ is marked,  is a constant independent of the value of the marking. 

\begin{remark}\label{remark:clique_regular} Note that when $H=K_r$ is the $r$-clique, for some $r \geq 2$, then $t_{a}(x,H,W)=t_{b}(x,H,W)$, for all $1 \leq a \ne b \leq r$. Hence, 
\eqref{eq:H_regular} simplifies to 
\begin{align}\label{eq:H_r}
t_{1}(x,K_r,W) = \mathbb{E}\left[\prod_{1\leq a < b \leq r}W(U_{a},U_{b})\biggm|U_{1}=x\right] = t(H,W), \text{ for almost every } x \in [0, 1],
\end{align}
which is precisely the notion of $K_{r}$-regularity defined in
\cite{hladky2019limit}. 
\end{remark}

\begin{remark}\label{RK2}
Recall that the 
 \emph{degree function} 
of a graphon $W$ is defined as
\begin{align}
  \label{dW}
d_{W}(x):=\int_{[0,1]}W(x,y) \ddy.
\end{align}
Note that for $H=K_2$, \eqref{ta} yields
\begin{align}
  \label{dw}
t_{1}(x,K_2,W) = \mathbb{E}\left[W(U_{1},U_{2})\bigm|U_{1}=x\right]  =
  \int_{[0,1]}W(x,y) \ddy = d_{W}(x).
\end{align}
Hence,
the notion of $K_2$-regularity
coincides with the standard notion of {\it degree regularity}, where the
 degree function $ d_{W}(x):=\int_{[0,1]}W(x,y) \ddy$ is constant
a.e.
\end{remark}

\subsection{Spectrum of Graphons and 2-Point Conditional Densities}\label{sec:aux}

Hereafter, we denote by $\mathcal{W}_0$ the space of all graphons, which is
the collection of all 
symmetric, measurable functions $W: [0, 1]^2\rightarrow [0, 1]$. 
We let also $\Wi$ be the  space of all bounded, 
symmetric, measurable functions $W: [0, 1]^2\rightarrow [0, \infty)$.
Every graphon $W\in \mathcal{W}_0$,
or more generally $W\in\Wi$, 
defines an operator $T_{W}:L^{2}[0,1]\rightarrow L^{2}[0,1]$ as follows:  
\begin{equation}
(T_Wf)(x)=\int_0^1W(x, y)f(y)\dd y, 
\label{eq:TW}
\end{equation} 
for each $f\in L^{2}[0,1]$. 
$T_W$ is a symmetric  Hilbert--Schmidt operator;
thus it  is compact and has a discrete spectrum, that is, it has a countable
multiset of non-zero real eigenvalues, which we denote by
$\mathrm{Spec}(W)$, with
\begin{align}\label{Spec}
\sum_{\lambda\in \Spec(W)}\lambda^2=\iint W(x,y)^2\dd x\dd y<\infty. 
\end{align}
Moreover, a.e.,
\begin{align}\label{TW}
(T_{W}f)(x)=\sum_{\lambda\in\mathrm{Spec}(W)}\lambda\langle f,\phi_{\lambda}\rangle\phi_{\lambda}(x)  
\end{align}
and 
\begin{align}\label{eq:Ws}
W(x,y)=\sum_{\lambda\in\mathrm{Spec}(W)}\lambda\phi_{\lambda}(x)\phi_{\lambda}(y),
\end{align}
 where $\{\phi_{\lambda}\}_{\lambda\in\mathrm{Spec}(W)}$ denotes an
 orthonormal system of eigenfunctions associated with
 $\mathrm{Spec}(W)$. For a more detailed discussion on the spectral
 properties of graphons and their role in graph limit theory, see
 \cite[Chapters 7, 11]{lovasz2012large}.

To describe the limiting distribution of $X_n(H, W)$ when $W$ is $H$-regular, we will need to understand the spectral properties of the following graphon obtained from the 2-point conditional densities: 

\begin{definition}\label{defn:WH} Given a graphon $W \in \mathcal{W}_0$ and
  a simple connected graph $H=(V(H), E(H))$, the {\it 2-point conditional
    graphon induced by $H$} is defined as:  
\begin{align}\label{eq:WH}
W_{H}(x,y)=\frac{1}{2|\Aut(H)|}\sum_{1\leq a\neq b\leq |V(H)|}t_{a, b}((x,y),H,W),
\end{align}
where $t_{a, b}((x,y),H,W)$ is the $2$-point conditional homomorphism
density function of $H$ in $W$ given the vertices $(a, b)$, as in Definition
\ref{defn:marked_density}.\footnote{Strictly speaking, $W_H$ is in general not a
  graphon in $\cW_0$ because it can take values greater than 1. 
However, $W_H\in \cW_1$, and we still call it a graphon.
}
(The normalization factor in \eqref{eq:WH} is chosen for later convenience
in e.g.~\eqref{tt2}.)
\end{definition}

Intuitively, $W_{H}(x,y)$ can be interpreted as the homomorphism density of
$H$ in $W$ containing the `vertices' $x, y \in [0, 1]$. 

Note that a graphon $W$ is $H$-regular (see Definition \ref{defn:regular}) 
if and only if the 2-point conditional graphon $W_H$ is
degree regular (see Remark \ref{RK2}). 
This is because, for all $x \in [0, 1]$,  
\begin{align}\label{eq:degreeWH}
\int_0^1 W_{H}(x,y)\dd y  
& = \frac{|V(H)|-1}{2|\Aut(H)|}\sum_{a=1}^{|V(H)|}t_{a}(x,H,W), 
\end{align}
and the RHS of \eqref{eq:degreeWH} is a constant if and only if $W$ is $H$-regular. In fact, if $W$ is $H$-regular, then $\frac{1}{|V(H)|} \sum_{a=1}^{|V(H)|}t_{a}(x,H,W) = t(H, W)$ a.e.; hence, the degree of $W_H$ becomes 
\begin{align}\label{eq:degreeH}
\int_0^1 W_{H}(x,y)\dd y = \frac{|V(H)|(|V(H)|-1)}{2|\Aut(H)|} \cdot t(H, W) & := d_{W_H}, 
\end{align} 
for almost every $x \in [0, 1]$. This implies that, if $W$ is $H$-regular,
then $d_{W_H}$ is an eigenvalue of the operator $T_{W_H}$ (recall
\eqref{eq:TW}) and $\phi \equiv 1 $ is a corresponding eigenvector. In this
case, we will use $\mathrm{Spec}^{-}(W_{H})$ to denote the collection
$\mathrm{Spec}(W_H)$ with the multiplicity of the eigenvalue $d_{W_H}$
decreased by $1$.
(Note that $d_{W_H}>0$ by \eqref{eq:degreeH} 
unless $t(H,W)=0$, or $|V(H)|=1$; these cases are both 
trivial, see Remark~\ref{Rtrivial}.)

\subsection{Statement of the Main Result}\label{sec:XHW} 

To state our results on the asymptotic distribution of $X_{n}(H,W)$, we need to define a few basic graph operations. 
\begin{definition}
For a graph $H=(V(H),E(H))$ on vertex set $\{1,2,\cdots, r\}$ define,
\begin{align}
	E^{+}(H)=\{(a,b):1\leq a\neq b\leq r, (a,b) \text{ or } (b,a)\in E(H)\}
\end{align}
\end{definition}

\begin{figure}
    \centering
    \includegraphics[scale=0.7]{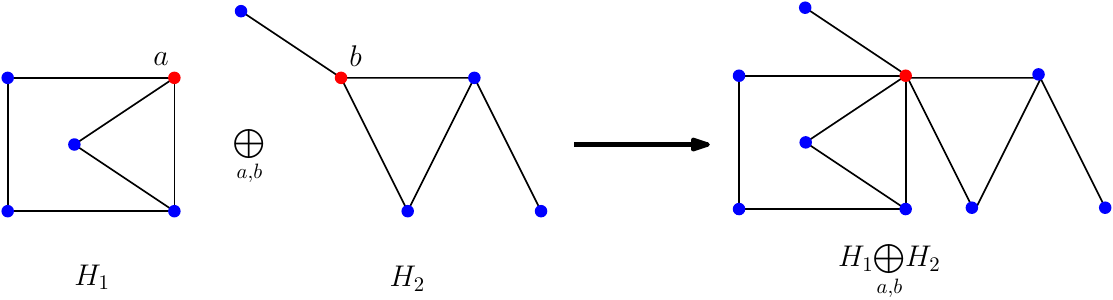}
    \caption{The $(a, b)$-{\it vertex join} of the graphs $H_1$ and $H_2$.}
    \label{fig:v_join}
\end{figure}

\begin{definition}\label{defn:H1H2ab} Fix $r \geq 1$ and consider two graphs $H_{1}$ and $H_{2}$ on the vertex set $\{1,2,\cdots,r\}$ and edge sets $E(H_1)$ and $E(H_2)$, respectively. 

\begin{itemize}

\item {\it Vertex Join}: For $a, b \in \{1,2,\cdots,r\}$, the $(a, b)$-{\it vertex join} of $H_1$ and $H_2$ is the graph obtained by identifying the $a$-th vertex of $H_1$ with the $b$-th vertex of $H_2$ (see Figure \ref{fig:v_join} for an illustration). The resulting graph will be denoted by 
\begin{align*}
H_{1}\bigoplus_{a,b}H_{2}.
\end{align*}

\item {\it Weak Edge Join}: For $(a,b) \in E^{+}(H_1)$ and $(c,d) \in E^{+}(H_2)$, with $1 \leq a\neq b\leq r$ and $1 \leq c\neq d \leq r$, the $(a, b), (c, d)$-{\it weak edge join} of $H_1$ and $H_2$ is the graph obtained identifying the vertices $a$ and $c$ and the vertices $b$ and $d$ and keeping a single edge between the two identified vertices (see Figure \ref{fig:e} for an illustration). The resulting graph will be denoted by 
\begin{align*}
    H_{1}\bigominus_{(a,b),(c,d)}  H_{2} . 
\end{align*}

\item {\it Strong Edge Join}: For $(a,b) \in E^{+}(H_1)$ and $(c,d) \in E^{+}(H_2)$, with $1 \leq a\neq b\leq r$ and $1 \leq c\neq d \leq r$, the $(a, b), (c, d)$-{\it strong edge join} of $H_1$ and $H_2$ is the multi-graph obtained identifying the vertices $a$ and $c$ and the vertices $b$ and $d$ and keeping both the edges between the two identified vertices (see Figure \ref{fig:e} for an illustration). The resulting graph will be denoted by 
\begin{align*}
H_{1}\bigoplus_{(a,b),(c,d)}  H_{2}.
\end{align*}

\end{itemize}
\end{definition}

\begin{remark}\label{rem:def_ext_join}
We note that 
both  the {weak} and {strong} edge join
operations can be extened to arbitrary $(a,b)\in V(H_{1})^{2}$ and $(c,d)\in
V(H_{2})^2$  with $ a\neq b$ and $ c\neq d$; in the strong join we keep all
edges, but in the weak join we keep the join simple by merging any resulting
double edge.
(Thus, if either $(a,b)\not\in E^{+}(H_{1})$ or $(c,d)\not\in
E^{+}(H_{2})$, then the {weak} and {strong} edge joins are the same graph.)  
\end{remark}

\begin{figure}
    \centering
    \includegraphics[scale=0.7]{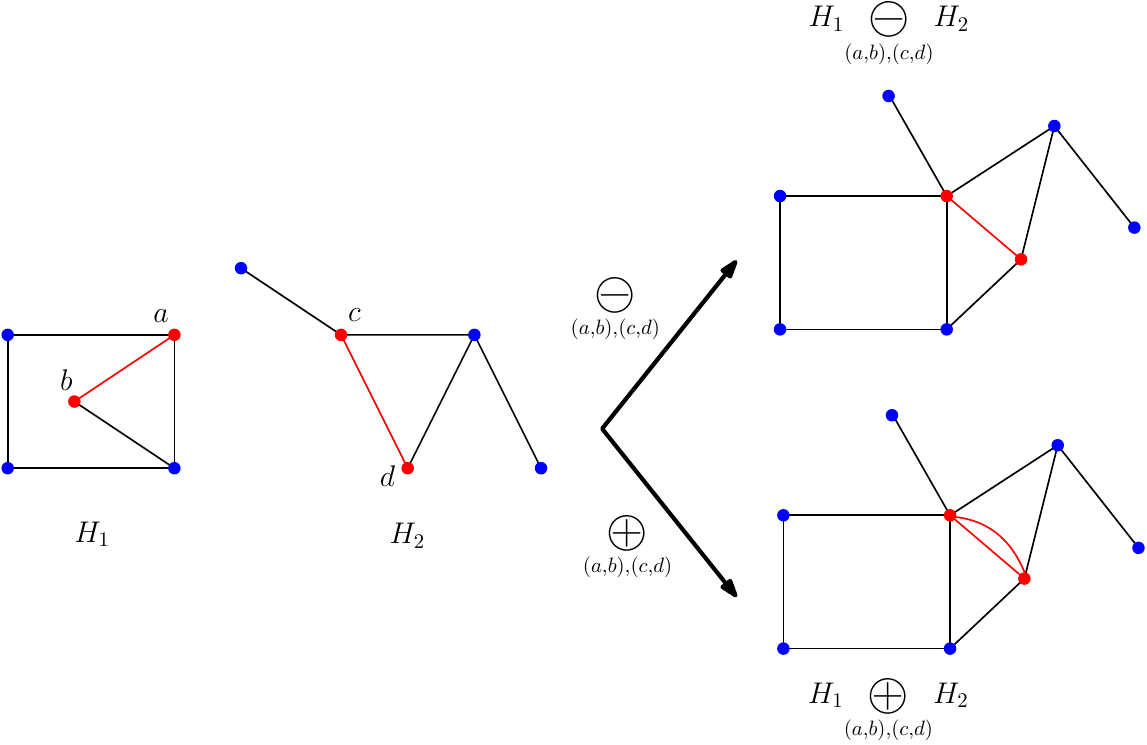}
    \caption{The weak and strong edge joins of the graphs $H_1$ and $H_2$.}
    \label{fig:e}
\end{figure}

Having introduced the framework and the relevant definitions, we are now
ready to state our main result regarding the asymptotic distribution of
$X_n(H, W)$, the number of copies of $H$ in the $W$-random graph $G(n,
W)$.

\begin{theorem}\label{thm:asymp_count}
Fix a graphon $W \in \mathcal{W}_0$ and a simple  graph
$H=(V(H), E(H))$ with vertices labeled $V(H)=\{1, 2, \ldots, |V(H)|\}$. Then
for $X_{n}(H,W)$ as defined in \eqref{eq:XHW} the following hold, as
$n\to\infty$:  
\begin{itemize}
    
    \item [$(1)$] For any $W$,
    \begin{align}\label{eq:limit1}
       \frac{X_{n}(H,W)-\frac{(n)_{|V(H)|}}{|\Aut(H)|}t(H,W)}{n^{|V(H)|-\frac{1}{2}}}\overset{D}{\longrightarrow}\mathsf{N}(0, \tau^{2}_{H, W}) , 
    \end{align}
    where 
    \begin{align}\label{ttau}
        \tau^{2}_{H, W} :=\frac{1}{|\Aut(H)|^{2}}\left[\sum_{1\leq a, b\leq
      |V(H)|}t\left(H\bigoplus_{a,b}H,W\right) - |V(H)|^{2}t(H,W)^{2}\right]
\ge0. 
    \end{align}
Moreover, $\tau^2_{H,W}>0$ if and only if 
W is not $H$-regular. Thus, 
if W is not $H$-regular, then $X_n(H,W)$ is asymptotically normal.   
    \item [$(2)$] If W is $H$-regular, then
    \begin{align}\label{tt2}
        \frac{X_{n}(H,W)-\frac{(n)_{|V(H)|}}{|\Aut(H)|}t(H,W)}{n^{|V(H)|-1}} \overset{D}{\longrightarrow}\sigma_{H,W} \cdot Z+\sum_{\lambda\in \mathrm{Spec}^{-}(W_{H})}\lambda(Z_{\lambda}^{2}-1) , 
    \end{align}
    where $Z$ and $\{Z_{\lambda}\}_{\lambda \in \mathrm{Spec}^{-}(W_{H})}$ all are independent standard Gaussians, 
    \begin{align}\label{gss}
        \sigma_{H,W}^{2}:=\frac{1}{2|\Aut(H)|^{2}}\sum_{(a,b),(c,d)\in
      E^{+}(H)}\left[t\left(H\bigominus_{(a,b),(c,d)}
      H,W\right)-t\left(H\bigoplus_{(a,b),(c,d)} H,W\right)\right]
\ge0, 
    \end{align}
    and $\mathrm{Spec}^{-}(W_{H})$ is the multiset $\mathrm{Spec}(W_H)$ with multiplicity of the eigenvalue $d_{W_H}$ $($recall \eqref{eq:degreeH}$)$ decreased by $1$. 
\end{itemize}
\end{theorem}

The sum in \eqref{tt2} may be infinite, but it converges in $L^2$ and
a.s.\ by \eqref{Spec}.  
The proof of Theorem \ref{thm:asymp_count} uses the projection method for
generalized 
$U$-statistics developed in Janson and Nowicki \cite{janson1991asymptotic},
which allows us to decompose $X_{n}(H,W)$ over sums of increasing
complexity. 
(See also  \cite[Chapter 11.3]{SJIII} and \cite{kaur2021higher}.) 
The terms in the expansion  are indexed by the vertices and
edges subgraphs of the complete graph of increasing sizes, and the
asymptotic behavior of $X_{n}(H,W)$ is determined by the non-zero terms
indexed by the smallest size graphs. Details of the proof are given in
Section \ref{sec:proof}. Various examples are discussed in Section
\ref{sec:example}.  

\begin{remark}\label{Rtrivial}
We note some trivial cases, where $X_n(H,W)$ is deterministic. 
First, $t(H,W)=1$ if and only if 
$H$ is empty (has no edge), or $W$ is {\it complete}, that is, $W\equiv 1$.
In both cases, almost surely
$X_{n}(H,W)=\frac{(n)_{|V(H)|}}{|\Aut(H)|}$.
Similarly, if $W$ is $H$-{\it free}, that is,  $t(H,W)=0$, then
almost surely $X_{n}(H,W)=0$. 
Note also that in these cases with $t(H,W)\in\set{0,1}$, we have
$\overline{t}(x,H,W) = t(H,W)$ a.e., e.g.\ by \eqref{intta},
and thus $W$ is $H$-regular. 
Theorem \ref{thm:asymp_count} is valid for these cases too (with
limits 0), 
but is not very interesting, and 
we may without loss of generality exclude these cases and assume $0<t(H,W)<1$.
\end{remark}

\begin{remark} As mentioned earlier, the result in Theorem
  \ref{thm:asymp_count}(1) has been  proved recently by F{\'e}ray,
  M{\'e}liot, and Nikeghbali  \cite[Theorem 21]{feray2020graphons} using
  the machinery of mod-Gaussian convergence. They noted that the limiting
  distribution in \cite[Theorem 21]{feray2020graphons} might be degenerate, that
  is, $\tau_{H, W}=0$, and called this case \emph{singular}.
(This is thus our $H$-regular case).
M\'eliot \cite{meliot2021central} 
studied the \emph{(globally) singular} graphons, i.e.,
the graphons $W$ for which $\tau_{H, W}=0$,
  {\it for all graphs} $H$. For such graphons \cite{meliot2021central} derived the order of fluctuations for the homomorphism
densities, but did not identify the limiting distribution. 
\end{remark}

The main emphasis of the present paper
is Theorem \ref{thm:asymp_count}(2),
for $H$-regular graphons, where the more interesting
  non-Gaussian fluctuation emerges. Moreover, it turns out that there are non-trivial cases where also the limit in
Theorem \ref{thm:asymp_count}(2) is degenerate.
We discuss this further in Section~\ref{Sdeg}, 
where we give both an example of such a higher-order degeneracy,
and examples of graphs $H$ for which this cannot happen for any $W$.
We will also study when one of the two components of the limit (the normal
and the non-normal component) vanishes.
In particular, in the classical Erd\H{o}s--R\'{e}nyi case $W\equiv p$,
Theorem \ref{thm:asymp_count}(2) applies to every $H$ with the non-normal
component vanishing, so the limit is normal, which is a classical result;
see further Example \ref{Ep}.

\begin{remark}
For the closely related problem of counting \emph{induced} subgraphs
isomorphic to $H$,   
limit distributions of the type in 
Theorem \ref{thm:asymp_count}(2) with a non-normal component occur (for special $H$) even in the Erd\H{o}s--R\'{e}nyi case $W\equiv p$,
but then with normalization by $n^{|V(H)|-2}$, 
see 
\cite{barbour1989central,janson1991asymptotic}. 
It seems interesting to study induced
subgraph counts in $G(n,W)$ for general graphons $W$ with our methods,
but we have not pursued this.
\end{remark}

Finally, it is worth mentioning that limiting distributions very similar to
that in Theorem \ref{thm:asymp_count}(2) also appears in the context of
counting monochromatic subgraphs in uniform random colorings of sequences of
dense graphs
~\cite{bhattacharya2017universal,bhattacharya2019monochromatic}. Although
this is a fundamentally different problem, the appearance of similar
limiting objects in both situations is interesting.

\section{Examples} 
\label{sec:example}

In this section we compute the limiting distribution of $X_{n}(H,W)$ for various specific choices of $H$ and $W$ using Theorem \ref{thm:asymp_count}.

\begin{example}(Cliques) Suppose $H=K_r$, the complete graph on $r$
  vertices, for some $r \geq 2$. This is the case that was studied in
  \cite{hladky2019limit}. To see that Theorem \ref{thm:asymp_count} indeed
  recovers the main result in \cite{hladky2019limit}, first recall Remark
  \ref{remark:clique_regular}, which shows that our notion of $H$-regularity
  matches with the notion of $K_r$-regularity defined in
  \cite{hladky2019limit}. Next, note that by the symmetry of the vertices of a clique,
\begin{align}
t\left(H\bigoplus_{a,b}H,W\right) = t\left(H\bigoplus_{1,1}H,W\right), 
\end{align} 
for $1\leq a, b\leq |V(H)|$, and $|\Aut(K_r)|= r!$. Therefore, Theorem \ref{thm:asymp_count}(1) implies, when $W$ is not $K_r$-regular,  
  \begin{align}
        \frac{X_{n}(K_r, W) - {n \choose r} t(K_r,W)}{n^{r-\frac{1}{2}}} \overset{D}{\rightarrow} \mathsf{N}\left(0, \frac{1}{(r-1)!^2}\left[ t\left(K_r \bigoplus_{1, 1} K_r, W\right) - t(K_r,W)^{2}\right] \right) ,
    \end{align}
which is precisely the result in \cite[Theorem 1.2(b)]{hladky2019limit}. For
the $K_r$-regular case, note that by the symmetry of the edges of a clique,  the
2-point conditional graphon induced by $K_r$ (recall Definition
\ref{defn:WH}) simplifies to
\begin{align}
W_{K_r}(x,y)=\frac{1}{2(r-2)!} \tij((x,y),K_r,W).
\end{align}
Moreover, for all $(a, b) , (c, d) \in E(K_r)$,
\begin{align}
t\left(K_r \bigominus_{(a,b),(c,d)}  K_r, W\right) = t\left(K_r \bigominus_{(1,2),(1,2)} K_r, W\right),
\end{align}
and similarly for the strong edge-join operation. Hence, Theorem \ref{thm:asymp_count}(2) implies
\begin{align}
   \frac{X_{n}(K_r, W) - {n \choose r} t(K_r,W)}{n^{r-1}}   \overset{D}{\rightarrow}\sigma_{K_r,W} \cdot Z+\sum_{\lambda\in \mathrm{Spec}^{-}(W_{K_r})}\lambda(Z_{\lambda}^{2}-1)
    \end{align}
    with \begin{align}
        \sigma_{K_r,W}^{2}=\frac{1}{2(r-2)!^{2}} \left\{t\left(H\bigominus_{(1,2),(1,2)} H,W\right)-t\left(H\bigoplus_{(1,2),(1,2)} H,W\right)\right\}, 
    \end{align}
as shown in \cite[Theorem 1.2(c)]{hladky2019limit}. 
\end{example}

\begin{figure}[!ht]
    \centering 
    \begin{minipage}{1.0\textwidth}
        \centering
        \includegraphics[width=4.75in]{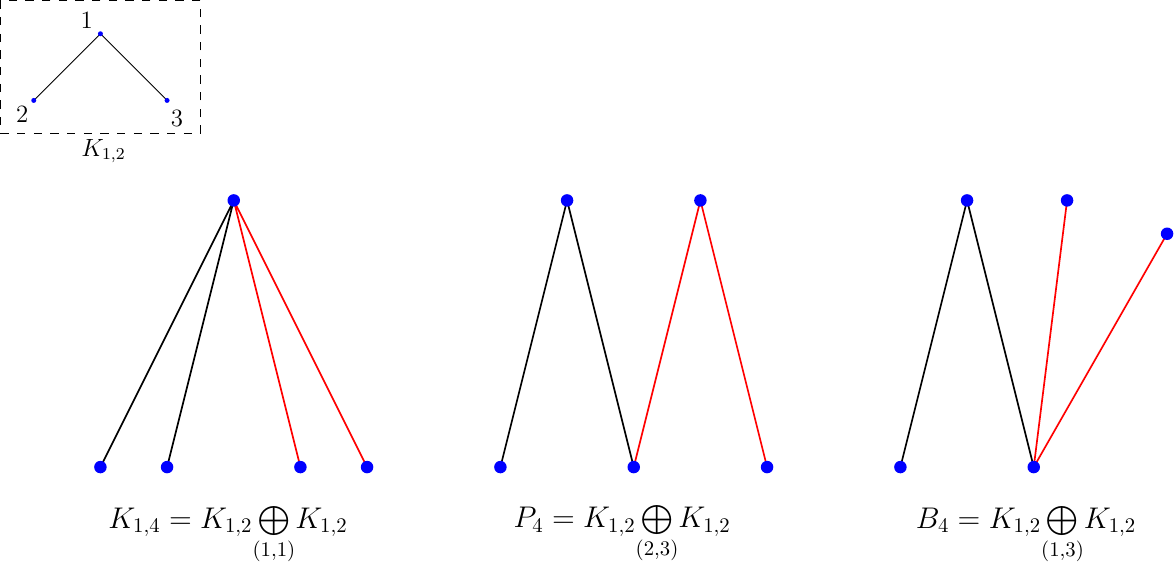} \\ 
            \end{minipage} 
            \caption{The different non-isomorphic graphs that can be obtained by the vertex join of two copies of $K_{1, 2}$ (with vertices labeled $\{1, 2, 3\}$ as in the inset). }
    \label{figure:example_K12}
\end{figure}

\begin{example} (2-Star) Suppose $H=K_{1,2}$ with the vertices labelled $\{1, 2, 3\}$ as shown in Figure \ref{figure:example_K12}. In this case, for any graphon $W \in \mathcal W_0$, 
\begin{align}\label{eq:t1_K12}
    t_{1}(x,K_{1,2},W)=\int_{0}^1 W(x,y)W(x,z)\dd y\dd z = d_{W}(x)^2, 
\end{align}
where the degree function $d_{W}(x)$ is defined in \eqref{dW}, 
and 
\begin{align}\label{eq:t23_K12}
    t_{2}(x,K_{1,2},W)=t_{3}(x,K_{1,2},W)=\int_0^1 W(x,y)W(y,z)\dd y\dd z = \int_0^1 W(x,y)d_{W}(y) \dd y , 
\end{align}    
Then by Definition \ref{defn:regular}, \eqref{eq:t1_K12} and
\eqref{eq:t23_K12}, $W$ is $K_{1,2}$-regular if and only if
\begin{align}\label{eq:K12_reg}
    d_{W}(x)^{2} + 2\int_{0}^{1}W(x,y)d_{W}(y)dy = 3t\bigpar{K_{1,2},W},
\quad\text{for a.e. }x\in [0,1]
.\end{align}
In particular, if $W$ is degree regular, then the left-hand side of 
\eqref{eq:K12_reg} is constant, and thus $W$ is $K_{1,2}$-regular.
(We conjecture that the converse holds too, but we have not verified this.)

Therefore, from Theorem \ref{thm:asymp_count} we have the following: 
\begin{itemize}
    
    \item If \eqref{eq:K12_reg} does not hold, then
        \begin{align}
       \frac{X_{n}(K_{1, 2},W)- 3 {n \choose 3} t(K_{1,2},W)}{n^{\frac{5}{2}}}\overset{D}{\rightarrow}\mathsf{N}(0, \tau^{2}_{K_{1, 2}, W})
    \end{align}
with \begin{align}
\tau^{2}_{K_{1, 2}, W} :=\frac{1}{4}\Big\{t(K_{1, 4},W) + 4 t(P_4, W) + 4 t(B_4, W) -  9 t(K_{1, 2},W)^{2} \Big\},
\end{align} 
where the graphs $K_{1, 4}$, $P_4$, and $B_4$ are as shown in Figure
\ref{figure:example_K12}. Note that $K_{1, 4}$ is the 4-star (obtained by
joining the two central vertices of the 2-stars), $P_4$ is the path with 4
edges (obtained by joining a leaf vertex of one 2-star with a leaf vertex of
another), and $B_4$ is the graph obtained by joining the central vertex of
one 2-star with a leaf vertex of another. For a concrete example of a
graphon which is not $K_{1, 2}$-regular, consider $W_0(x,y):=xy$. In this
case, $d_{W_0}(x) = \frac12x$, 
for all  $x\in [0,1]$, 
and \eqref{eq:K12_reg} does not hold;
hence, $W_0$ is not $K_{1,2}$-regular.

\begin{figure}[!ht]
    \centering
    \begin{minipage}{1.0\textwidth}
        \centering
        \includegraphics[width=4.35in]{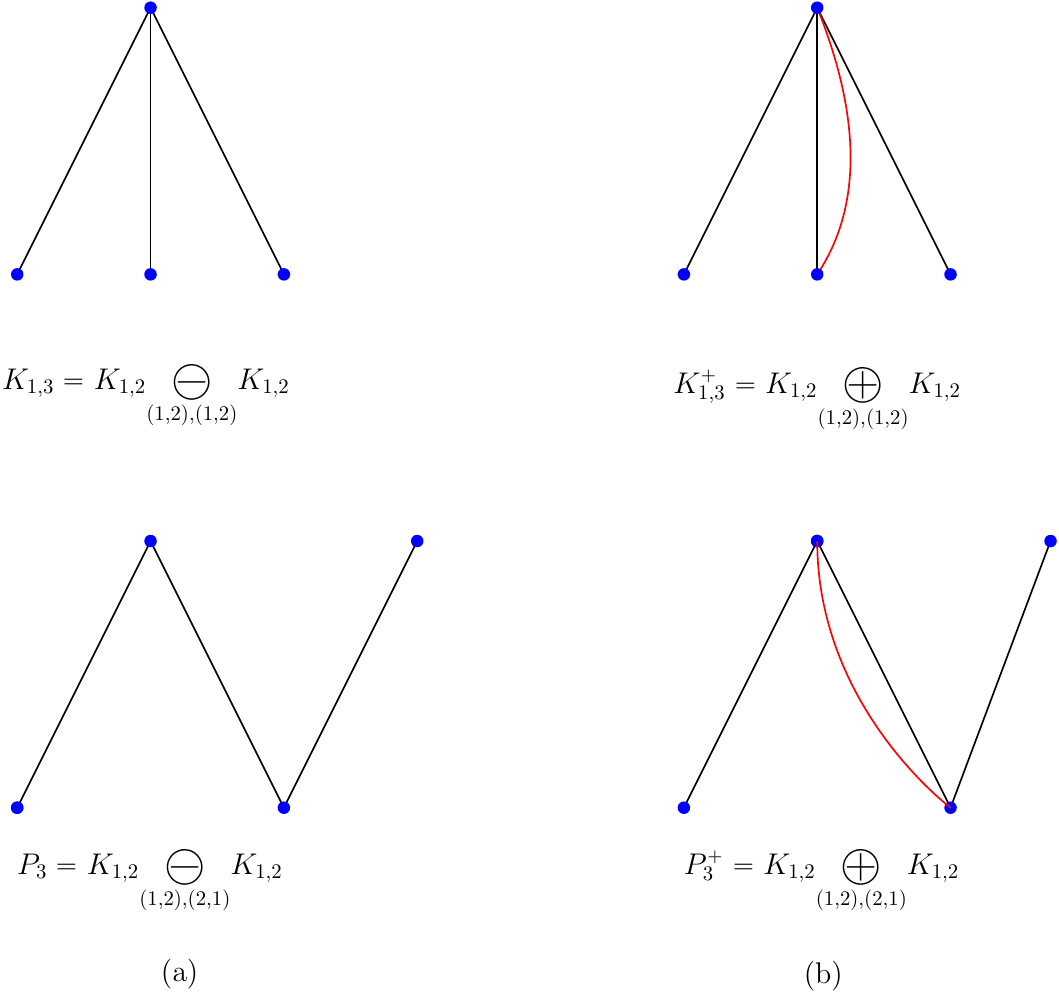} \\ 
            \end{minipage} 
            \caption{(a) The weak edge join of two copies of $K_{1,2}$ and (b) the strong edge join of two copies of $K_{1,2}$. }
    \label{figure:example_edge_K12}
\end{figure}

\item For every $H$, 
\begin{align}
        \frac{X_{n}(K_{1, 2},W)- 3 {n \choose 3} t(K_{1,2},W)}{n^{2}} \overset{D}{\rightarrow}\sigma_{K_{1, 2},W} \cdot Z+\sum_{\lambda\in \mathrm{Spec}^{-}(W_{K_{1, 2}})}\lambda(Z_{\lambda}^{2}-1) , 
    \end{align}
with 
\begin{align}\label{eq:sigma_K12}
	\sigma_{K_{1, 2},W}^{2}:=2\{t(K_{1, 3}, W )+t(P_3,W)-t(K_{1, 3}^+,W)-t(P_3^+,W)\},
\end{align} 
where $K_{1, 3}$ is the 3-star and $P_3$ is the path 
shown in Figure \ref{figure:example_edge_K12}(a) 
(obtained by the weak edge-join of two
copies of $K_{1, 2}$ using the edges $(1,2),(1,2)$ and $(1,2),(2,1)$
respectively) and the $K_{1, 3}^{+}$  and $P_3^{+}$ are the multigraphs
shown in Figure \ref{figure:example_edge_K12}(b)  (obtained by the strong
edge-join of two copies of $K_{1, 2}$  using the edges $(1,2),(1,2)$ and
$(1,2),(2,1)$ respectively). Moreover, in this case the 2-point conditional
graphon $W_{K_{1, 2}}$ simplifies to:  
\begin{align}\label{eq:K12W}
W_{K_{1, 2}}(x, y)= \frac{1}{2} \left\{ W(x, y)( d_W(x) + d_W(y) ) + \int W(x, z) W(y, z) \dd z \right \},
\end{align}
since $\tij(x, y, K_{1, 2}, W) = t_{1, 3}(x, y, K_{1, 2}, W)= W(x, y) d_W(x)$ and $t_{2, 3}(x, y, K_{1, 2}, W) = \int_{[0, 1]} W(x, z) W(y, z)  \dd z$, and similarly for the others. 
For a concrete example of graphon which is $K_{1, 2}$-regular consider 
\begin{align}
    \tilde{W}(x,y):=
    \begin{cases}
        p & \text{ if }(x,y)\in \left[0, \frac{1}{2}\right]^2 \bigcup \left[\frac{1}{2}, 1\right]^2 , \\
        0 & \text{ otherwise} .
    \end{cases}
\end{align} 
Note that this is a 2-block graphon (with equal block sizes) taking value $p$ in the diagonal blocks and zero in the off-diagonal blocks. (One can think of this as the `disjoint union two Erd\H{o}s--R\'{e}nyi graphons'.) 
It is easy to check that this graphon is degree regular, hence $K_{1, 2}$-regular. In fact, in this case
\begin{align}
    \tilde{W}_{K_{1,2}}(x,y)=
    \begin{cases}
        \frac{3p^{2}}{4} & \text{ if }(x,y)\in \left[0, \frac{1}{2}\right]^2 \bigcup \left[\frac{1}{2}, 1\right]^2 , \\
        0 & \text{ otherwise} .
    \end{cases}
\end{align}
and $\sigma_{K_{1, 2},\tilde{W}}^{2} =\frac{1}{2}p^3(1-p)$.
Moreover,
\begin{align}
\mathrm{Spec}(\tilde{W}_{K_{1, 2}})=\{3p^{2}/8,3p^{2}/8\},
\end{align} with the
eigenfunctions $1$ and $\bm{1}\{[0,1/2]\} - \bm{1}\{[1/2,1]\}$,
respectively. 
In particular, $d_{W_{K_{1,2}}}=3p^2/8$
in agreement with \eqref{eq:degreeH}.
Consequently, $\mathrm{Spec}^{-}(\tilde{W}_{K_{1, 2}})=\{3p^{2}/8\}$.  
\end{itemize} 
\end{example}

\begin{example}(Erd\H{o}s--R\'{e}nyi graphs)\label{Ep}
Suppose that $W=W_p\equiv p$ for some $p\in(0,1)$.
By symmetry,
$ \overline{t}(x,H,W)$ does not depend on $x$, and thus $W_p$ 
is $H$-regular for every $H$. 
Furthermore, by \eqref{eq:WH}, also the 2-point conditional graphon $W_H$ is
constant, which implies (see also Proposition \ref{ppn2}) that 
$\mathrm{Spec}^-(W_H)=\emptyset$ and thus the limit in 
Theorem \ref{thm:asymp_count}(2) is normal for every non-empty $H$.
(We have $\sigma^2_{H,W}>0$ by \eqref{gss}.)
As said earlier, this is a classical result, see e.g.\
\cite{Nowicki1989,nowicki1988subgraph,rucinski1988small,barbour1989central,SJ79,SJ94,JLR}.
\end{example}

\section{Degeneracies of the Asymptotic Distribution}
\label{Sdeg}

In this section we will discuss the degeneracies of asymptotic distribution
when $W$ in $H$-regular; we will throughout the section tacitly 
ignoring the trivial cases in Remark \ref{Rtrivial}, i.e., we
assume that $0<t(H,W)<1$.
Towards this denote
\begin{align}
       Z_n(H, W) :=
        \frac{X_{n}(H,W)-\frac{(n)_{|V(H)|}}{|\Aut(H)|}t(H,W)}{n^{|V(H)|-1}} . 
\end{align}        
Theorem \ref{thm:asymp_count}(2) shows that when $W$ is $H$-regular, 
\begin{align}\label{eq:limit2}
   Z_n(H, W)      \overset{D}{\rightarrow}\sigma_{H,W} \cdot Z
+\sum_{\lambda\in \mathrm{Spec}^{-}(W_{H})}\lambda(Z_{\lambda}^{2}-1) , 
    \end{align}
    where $Z, \{Z_{\lambda}\}_{\lambda \in \mathrm{Spec}^{-}(W_{H})}$ are all independent standard Gaussians, and $\sigma_{H,W}^{2}$ is as defined in Theorem \ref{thm:asymp_count}. This raises the following natural questions: 
       
\begin{itemize}

\item {\it Is the limiting distribution of $Z_n(H, W)$ non-degenerate?}
  Given the result in Theorem~\ref{thm:asymp_count} it is natural to wonder
  whether, when $W$ is $H$-regular,  
the limiting distribution of $Z_n(H, W)$ in \eqref{eq:limit2}
is always  non-degenerate. This is indeed the case for cliques: 
if $H=K_r$ for some $r \geq 2$, then it 
was shown in  \cite[Remark 1.6]{hladky2019limit} 
that the limit in  \eqref{eq:limit2} is never degenerate.
However, for general graphs $H$ the situation is
  surprisingly more complicated. It turns out that there are graphs $H$ for
  which there exist a $H$-regular graphon $W$, with $0<t(H, W)<1$, such that
  the limit in \eqref{eq:limit2} is degenerate (see Example
  \ref{example:3star}). Naturally this raises the question: {\it For which
    graphs $H$ is the limiting distribution of $Z_n(H, W)$ always
    non-degenerate?} In Section \ref{sec:cyclestar} we answer this question
 in the affirmative when $H=C_4$ is the 4-cycle and $H=K_{1, 2}$ is the 2-star. 

\end{itemize}

In cases when the limit in \eqref{eq:limit2} is non-degenerate, we can ask
about the structure of $W$ when one of the components of the limit vanishes:

\begin{itemize}

\item {\it When is the limiting distribution of $Z_n(H, W)$ normal?} Note from \eqref{eq:limit2} that $Z_n(H, W)$ is asymptotically Gaussian if and only if the non-Gaussian component
\begin{align*}
\sum_{\lambda\in \mathrm{Spec}^{-}(W_{H})}\lambda(Z_{\lambda}^{2}-1) 
\end{align*}
is degenerate. We show in Proposition \ref{ppn2} that this happens precisely
when the 2-point conditional graphon $W_H$ is constant a.e.  

\item {\it When is the limiting distribution of $Z_n(H, W)$ normal-free?}
  Clearly, the limit \eqref{eq:limit2} has no Gaussian component whenever
  $\sigma_{H, W}=0$. In Theorem \ref{thm:bizero} we characterize the
  structure of such graphons when $H$ is bipartite: we show that if $H$ is
  bipartite,  then the limit in \eqref{eq:limit2} is normal-free if and only if
  $W(x,y)\in \{0,1\}$ a.e. (that is,  $W$ is {\it random-free}).  
We also show that there are non-bipartite graphs $H$ and graphons $W$ which are 
not random-free for which $\sigma_{H, W}=0$ (Example \ref{ExH}).

\end{itemize}

\subsection{Degeneracy of the Non-Gaussian Component}

The following proposition characterizes when the limit in \eqref{eq:limit2}
is Gaussian. 
It extends the special case $H=K_r$ which was shown in  
\cite[Theorem 1.3]{hladky2019limit}.

\begin{prop}\label{ppn2}
Let\/ $H$ be a simple graph and let\/ $W$ be a $H$-regular graphon. Then the following are equivalent: 
\begin{enumerate}

\item[$(1)$] $Z_n(H, W) \stackrel{D}\rightarrow N(0, \sigma^2_{H, W})$.

\item[$(2)$] $\sum_{\lambda\in \mathrm{Spec}^{-}(W_{H})}\lambda(Z_{\lambda}^{2}-1)$ is degenerate.

\item[$(3)$] $\Spec^-(W_{C_4})=\emptyset$.

\item[$(4)$]
$W_H(x,y)=d_{W_H}$ a.e., where $d_{W_H}=\frac{|V(H)|(|V(H)|-1)}{2 |\Aut(H)|}\cdot t(H,W)$ is as defined in \eqref{eq:degreeH}. 
\end{enumerate}
\end{prop}

\begin{proof} From \eqref{eq:limit2} it is clear that (1), (2) and (3) are
  equivalent. Next,
recalling the discussion following  \eqref{eq:degreeH},
 $\Spec^-(W_H)=\emptyset$ if and only if
  $\Spec(W_H)=\{d_{W_H}\}$;
furthermore, 
since $W$ is $H$-regular, $W_H$ is degree regular
  and, hence, $\phi\equiv1$ is an eigenfunction corresponding to $d_{W_H}$. 
Therefore, by \eqref{eq:Ws}, 
if   $\Spec(W_H)=\{d_{W_H}\}$, then
\begin{align}
  W_H(x,y)= d_{W_H}\phi(x)\phi(y) = d_{W_H} 
\qquad\text{a.e.}
\end{align}
Conversely, $W_H(x,y)=d_{W_H}$ a.e.\
implies that $d_{W_H}$ is the only non-zero
eigenvalue of $T_{W_H}$, and thus $\Spec^-(W_H)=\emptyset$. This establishes
that (3) and (4) are equivalent.  
\end{proof}

\subsection{Degeneracy of the Gaussian Component}

The Gaussian component in the limit \eqref{eq:limit2} is degenerate when $\sigma_{H,W}^2 = 0$. To study the structure of such graphons we need a few definitions.  For a graph $F=(V(F), E(F))$ and $S\subseteq V(F)$, the  {\it neighborhood} of $S$ in $F$ is $N_F(S)=\{v \in V(F): \exists ~u \in S \text{ such that } (u, v) \in E(F)\}$. Moreover, for $u, v \in V(F)$, $F\backslash \{u, v\}$ is the graph obtained by removing the vertices $u, v$ and all the edges incident on them. For notational convenience we introduce the following definition: 

\begin{definition}\label{defn:tabxyHW}
 Let $H$ be a labeled finite simple graph and $W$ a graphon. Then, for $1 \leq u \ne v \leq |V(H)|$, the function $t_{u, v}^{-}(\cdot, \cdot, H, W): [0, 1]^2 \rightarrow [0, 1]$ is defined as: 
\begin{align}\label{tab-}
& t_{u, v}^{-}(x, y, H, W) \notag\\
&= \int_{[0,1]^{|V(H)|-2}} \prod_{r \in N_H(u)\backslash \{v\}}W(x, z_r) \prod_{s \in N_H(v) \backslash \{u\}} W(y, z_s) \prod_{(r, s) \in E(H\backslash \{u, v\}) } W(z_r, z_s) \prod_{r \notin \{u, v\}} \dd z_r . 
\end{align} 
\end{definition}

Thus, if $(u,v)\in E(H)$, then
\begin{align}
  \label{tab-tab}
t_{u, v}(x, y, H, W) = W(x,y) t_{u, v}^{-}(x, y, H, W) .
\end{align}

Note that 
\begin{align}\label{eq:sigma-sq-exp} 
 \sigma_{H,W}^2   & = c_H\sum_{(a,b),(c,d)\in E^{+}(H)}
    \int t_{a, b}^{-}(x, y, H, W) t_{c, d}^{-}(x, y, H, W) W(x,y)(1-W(x,y)) \ddx \ddy , 
\end{align}    
where $c_H:= \frac{1}{2 |\Aut (H)|^2}$. It is clear from
\eqref{eq:sigma-sq-exp} that if $W$ is random free, then $\sigma_{H,W}^2=0$
and hence, if $W$ is $H$-regular, the asymptotic distribution does not have
a normal 
component. Interestingly, the converse is also true whenever $H$ is
bipartite. This is formulated in the following theorem:  


\begin{theorem}\label{thm:bizero}
If $H$ is a non-empty
bipartite graph with $t(H,W)>0$, then $\sigma_{H,W}^2 = 0$ if
and only if  $W$ is random-free.
\end{theorem}

The proof of Theorem \ref{thm:bizero} is given in Section
\ref{sec:bipartitepf}. It entails showing,  using the bipartite structure of
$H$, that for almost every $(x, y)$ such that $W(x, y) \in (0, 1)$, 
we have $t_{a, b}^{-}(x, y, H, W) > 0$, for $a\neq b\in V(H)$ such that 
$(a,b)\in E(H)$. Consequently, from \eqref{eq:sigma-sq-exp},  
$\sigma_{H,W}^2 > 0$ whenever the set $\{(x, y) \in [0, 1]^2: W(x, y) \in
(0, 1)\}$ has positive Lebesgue measure. An immediate consequence of Theorem
\ref{thm:bizero} is that for a bipartite graph $H$
and an $H$-regular $W$, 
the asymptotic distribution of $Z_n(H, W)$ is non-degenerate  whenever $W$ is not random free.

\begin{remark}\label{remark:biHW}  
The bipartite assumption in Theorem \ref{thm:bizero} is necessary, in the sense that there exist 
non-bipartite graphs $H$ and graphons $W$ with
$t(H,W)>0$ such that $\sigma_{H,W}^2 = 0$, but $W$ is not random-free. We
discuss this in Example \ref{ExH}.
\end{remark} 

For non-bipartite $H$, we note only the following, which extends 
\cite[Proposition 1.5]{hladky2019limit}.

\begin{prop}
  We have $\gss_{H,W}=0$ if and only if\/ $W(x,y)=1$ for a.e.\ $(x,y)$ such
  that 
$t_{a, b}(x, y, H, W)>0$ for some $(a,b)\in E^+(H)$.
\end{prop}
\begin{proof}
An immediate consequence of \eqref{eq:sigma-sq-exp} and \eqref{tab-tab}.  
\end{proof}

\subsection{Degeneracy of the Limit in (\ref{eq:limit2})} 
\label{sec:cyclestar}

We begin with an example where the limit in \eqref{eq:limit2} is degenerate. 

\begin{example}\label{example:3star}
Let $H=K_{1,3}$ be the 3-star on vertex set $\{1,2,3,4\}$, where the  root
node is labeled $1$. 
Further, suppose that $W$ is the complete bipartite graphon: 
\begin{align}\label{eq:bipartiteW}
    W(x,y):=
    \begin{cases}
        0 & \text{ if }(x,y)\in \left[0, \frac{1}{2}\right]^2 \bigcup \left(\frac{1}{2}, 1\right]^2 , \\
        1 & \text{ otherwise} .
    \end{cases}
\end{align} 
To begin with note that $d_{W}(x)=\int_{0}^{1}W(x,y) \ddy=\frac{1}{2}$, for all $x\in [0,1]$. Therefore, 
\begin{align}
    \frac{1}{4}\sum_{i=1}^{4}t_{i}\left(x,K_{1,3},W\right)
    &=\frac{1}{4}\left[d_W(x)^3+3\int W(x,t)d_W(t)^2 \dd t\right]=  \frac{1}{8} .
\end{align}
This establishes that $W$ is $K_{1,3}$-regular, and that $t(K_{1,3},W)=1/8$. 
Next, since $W \in \{0,1\}$,
by Theorem \ref{thm:bizero}, $\sigma_{K_{1,3},W_{2}}^{2}=0$. Hence, to show
that the limit distribution of $Z_n(K_{1, 3}, W)$ is degenerate it
suffices to check that  
$\sum_{\lambda\in\text{Spec}^{-}(W_{K_{1,3}})}\lambda^{2} = 0 .$ 
By Proposition \ref{ppn2}, this is equivalent to showing
\begin{align}\label{eq:degencond2}
    W_{K_{1,3}}(x,y) = \frac{12}{ 2\left|\mathrm{\Aut}(K_{1,3})\right|}t\left(K_{1,3},W\right) = \frac{1}{8}, 
\end{align}
for a.e. $(x,y)\in [0,1]^2$ (since $\left|\mathrm{\Aut}(K_{1,3})\right| = 3!
= 6$).  Towards this recall \eqref{eq:WH}, which yields
\begin{align}\label{eq:Wxystar}
    W_{K_{1,3}}(x,y) & =\frac{1}{2|\Aut(K_{1,3})|}\sum_{1 \leq a\neq b \leq 4}t_{a,b}\left(x,y,K_{1,3},W\right) \nonumber \\ 
    & =\frac{1}{12}\bigg[3W(x,y)\int W(x,z)W(x,t) \dd z \dd t + 3W(x,y)\int W(y,z)W(y,t)\dd z \dd t  \nonumber \\
    & \hskip 5em + 6\int W(x,t)W(y,t)W(z,t) \dd z \dd t\bigg]  \nonumber \\
    &=\frac{1}{12}\left[3W(x,y)d_W(x)^2 + 3W(x, y)d_W(y)^2 + 6 \int d_W(t) W(x,t)W(y, t) \dd t\right] \nonumber \\
    &=\frac{1}{12}\left[\frac{3}{2}W(x,y) + 3\int W(x,t)W(y, t) \dd t\right]. 
\end{align}
Now, observe that if $W(x,y) = 0$ then $\int W(x,t)W(y, t) \dd t =
\frac{1}{2}$, which implies, from \eqref{eq:Wxystar}, $W_{K_{1,3}}(x,y) =
\frac{1}{8}$. Further, when $W(x,y) = 1$, then $\int W(x,t)W(y, t) \dd t =
0$, which implies $W_{K_{1,3}}(x,y) = \frac{1}{8}$. Thus for all $(x,y)\in
[0,1]^2$, $W_{K_{1,3}}=1/8$, which establishes \eqref{eq:degencond2}. This
shows that limiting distribution of $Z_n(K_{1, 3}, W)$ is degenerate for $W$
as in \eqref{eq:bipartiteW}.  

In fact,
in this example, we can easily find the asymptotic distribution of
$W_{K_{1,3}}$ directly.
Let
$M:=\bigl|\set{i:U_i\le\frac12}\bigr|\sim\operatorname{Bin}\bigpar{n,\frac12}$,
and $\tM:=M-n/2$. Then
\begin{align}\label{K13a}
  X_n(K_{1,3},W)&=M\binom{n-M}3 + (n-M)\binom M3
\notag\\&
=\frac{1}{6}\Bigpar{M(n-M)\bigpar{(n-M)^2-3(n-M)+2}+(n-M)M\bigpar{M^2-3M+2}}
\notag\\&
=\frac{1}{6}M(n-M)\bigpar{(n-M)^2+M^2-3n+4}
\notag\\&
=\frac{1}{6}\Bigpar{\frac{n}2+\tM}\Bigpar{\frac{n}2-\tM}
\Bigpar{\Bigpar{\frac{n}2-\tM}^2+\Bigpar{\frac{n}2+\tM}^2-3n+4}
\notag\\&
=\frac{1}{6}\Bigpar{\Bigpar{\frac{n}2}^2-\tM^2}
\Bigpar{2\Bigpar{\frac{n}2}^2+2\tM^2 -3n+4}
\notag\\&
=\frac{1}{3}\Bigpar{\Bigpar{\frac{n}2}^4-\tM^4}
-\frac{3n-4}{6}\Bigpar{\Bigpar{\frac{n}2}^2-\tM^2}.
\end{align}
Hence, subtracting the mean and using \eqref{EXHW},
\begin{align}\label{K13tM}
\frac{X_n(K_{1,3},W)-\frac{(n)_4}{48}}{n^2}&
=-\frac{\tM^4-\E \tM^4}{3n^2}
+\frac{3n-4}{6n}\cdot\frac{\tM^2-\E\tM^2}{n}.
\end{align}
Since the central limit theorem yields $\tM/n\qq\dto Z/2$, with all moments,
where $Z\sim N(0,1)$, 
\eqref{K13tM} yields
\begin{align}\label{K13tMa}
\frac{X_n(K_{1,3},W)-\frac{(n)_4}{48}}{n^2}&
\dtoo
-\frac{Z^4-3}{48}
+\frac{Z^2-1}8
=-\frac{1}{48}\bigpar{Z^4-6 Z^2+3}
=-\frac{1}{48}h_4(Z)
,\end{align}
where $h_4$ is the 4th Hermite polynomial (using the normalization in e.g.\
\cite[Example 3.18]{SJIII}).
Consequently, in this example, the correct normalization is by
$n^2=n^{|V(H)|-2}$, and the limit distribution is given by a 
fourth-degree polynomial of a Gaussian variable.
\end{example} 

The example above raises the question for which graphs $H$ is the limiting
distribution of $Z_n(H, W)$ 
in Theorem \ref{thm:asymp_count}(2)
non-degenerate for all graphons $W$. In the
following we will show that the limit is always non-degenerate when
$H=C_{4}$ or $H=K_{1, 2}$ (the 4-cycle and the 2-star). 
Our proofs use the specific
structure of the 4-cycle and 2-star and it remains unclear for what other
graphs can one expect the non-degeneracy result to hold.

\medskip 
  
\noindent {\it Non-Degeneracy of the Limit for the 4-Cycle}:  We begin by
deriving explicit conditions for degeneracy of the two 
components of the limiting distribution of $Z(C_4, W)$. 
(For the normal part, we can also use Theorem \ref{thm:bizero}, but we find
it interesting to first make a direct evaluation of the condition
$\gss_{H,W}=0$.) 
Towards this define:  
\begin{align}\label{eq:Uxy}
U_1(x,y):=\int_{[0, 1]} W(x,s)W(y,s) \dd s  \text{\quad and\quad } 
U_{2}(x,y):=\int_{[0, 1]^2} W(x,s)W(s,t)W(y,t) \dd s \dd t.
\end{align}

\begin{lemma}\label{LC4}
Suppose $W$ is a $C_{4}$-regular graphon with $t(C_{4},W)>0$. 
Then the following hold: 

\begin{alphenumerate}

\item $\Spec^-(W_{C_4})=\emptyset$ if and only if 
\begin{align}\label{eq:condition1}
U_1(x, y)^2 + 2W(x,y)U_{2}(x,y) = 3 t(C_4, W), 
\quad\text{a.e. } (x, y) \in [0,1]^2. 
\end{align}

\item $\sigma^2_{C_4,W}=0$ if and only if 
\begin{align}\label{eq:condition2}
\int_{[0, 1]^2} U_2^{2}(x,y) \left(W(x, y) - W^{2}(x,y) \right) \ddx \ddy =0. 
\end{align} 

\end{alphenumerate}
As a consequence, the limit of $Z_{n}(C_{4},W)$ in \eqref{eq:limit2} is degenerate if and only if  \eqref{eq:condition1} and \eqref{eq:condition2} hold. 
\end{lemma}

\begin{proof}
 Since all the vertices of the 4-cycle are symmetric, from Definition \ref{defn:regular} we have the following: The graphon $W$ is
$C_4$-regular if    
\begin{align}
    \int_{[0, 1]^3} W(x,y)W(y,z)W(z,t)W(t,x) \ddy \dd z \dd t=t(C_4, W) \text{ a.e. } x \in [0, 1] . 
\end{align}
Moreover, since $|\Aut(C_4)|=8$,  by Definition \ref{defn:WH}, 
the {2-point conditional graphon induced by $C_4$} is given by 
\begin{align}
W_{C_4}(x, y)=\frac{ 4 U_1(x, y)^2 + 8 W(x, y) U_2(x, y) }{2 |\Aut(C_4)|} = \frac{ U_1(x, y)^2 + 2 W(x, y) U_2(x, y) }{4},  
\end{align} 
where $U_1, U_2$ are as defined in \eqref{eq:Uxy}.
Hence, Proposition \ref{ppn2} shows that $\Spec^-(W_{C_4})=\emptyset$ 
if and only if \eqref{eq:condition1} holds.

Next, since all the edges of $C_4$ are symmetric, the weak edge join of 2 copies of $C_4$ is always isomorphic to graph $F_1$ in Figure \ref{fig:joins4}(a). Similarly, the strong edge join of 2 copies of $C_4$ is always isomorphic to graph $F_2$ in Figure \ref{fig:joins4}(b). 
\begin{figure}[!h]
    \centering
    \includegraphics{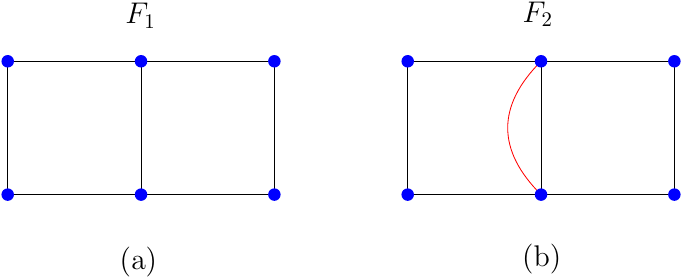}
    \caption{(a) The weak and (b) the strong edge join of two copies of $C_4$.}
    \label{fig:joins4}
\end{figure} 
Therefore, using $|E^+(C_4)|=8$ and $|\Aut(C_4)|=8$ in \eqref{gss}, 
we find that $ \sigma^2_{C_4, W}$ simplifies to 
\begin{align}
 \sigma^2_{C_4, W} & = \frac{1}{2} \left( t(F_1, W) - t(F_2, W) \right) \nonumber \\ 
  &  = \frac{1}{2} \left(\int_{[0, 1]^2} W(x,y)U_2^{2}(x,y) \dd x \ddy -\int_{[0, 1]^2} W^{2}(x,y)U_2^{2}(x,y) \dd x \ddy \right). 
\end{align}
Hence, 
\begin{align}
\sigma^2_{C_4, W} = 0 \iff \int_{[0,1]^2} U_2^{2}(x,y) \left ( W(x, y) - W^{2}(x,y) \right) \dd x \ddy ,
\end{align} 
which completes the proof. 
\end{proof}

The following theorem shows that
(if we ignore the trivial cases in Remark \ref{Rtrivial}),
whenever $W$ is $C_{4}$-regular, 
the limiting distribution of $Z_n(C_4, W)$ is always
non-degenerate. Hence, for $H=C_4$, 
Theorem \ref{thm:asymp_count}(1) or (2) will give a non-degenerate limit.
By Lemma \ref{LC4}, Theorem \ref{thm:condition} 
is equivalent to the claim that
whenever $W$ is $C_{4}$-regular, 
\eqref{eq:condition1} and \eqref{eq:condition2} cannot occur
simultaneously. 
The proof of Theorem \ref{thm:condition} is given in Section \ref{sec:conditionpf}.

\begin{theorem}\label{thm:condition}
Suppose $W$ is a $C_{4}$-regular graphon with $t(C_{4},W)>0$ and\/ $W$ is not
identically $1$ a.e. 
Then,
the limit of $Z_{n}(C_{4},W)$ in \eqref{eq:limit2} is non-degenerate.
\end{theorem}

\medskip 

\noindent {\it Non-Degeneracy of the Limit for the 2-Star}:  As in Lemma
\ref{LC4}, we first derive conditions which are equivalent to degeneracy of
the two components of the
limiting distribution of $Z_n(K_{1, 2}, W)$.

\begin{lemma}\label{lem:LK12}

Suppose $W$ is a $K_{1, 2}$-regular graphon with $t(K_{1, 2},W)>0$
Then the following hold: 

\begin{alphenumerate}

\item $\Spec^-(W_{K_{1, 2}})=\emptyset$ if and only if 
\begin{align}
W(x,y)\left(d_{W}(x)+d_{W}(y)\right) + U_1(x, y)  = 3\int d_{W}^{2}(z)\dd z,  
\quad\text{a.e. }(x,y)\in [0,1]^{2} , \label{eq:K12-cond1}
\end{align} 
where $U_1(x, y)$ is as defined in \eqref{eq:Uxy}. 

\item $\sigma^2_{K_{1, 2}, W}=0$ if and only if 
\begin{align}
\int \left\{ d_W(x) d_W(y) + d_W(x)^2 \right\} W(x,y)(1-W(x,y))\dd x \dd y = 0 . \label{eq:K12-cond2}
\end{align} 

\end{alphenumerate}
As a consequence, the limit of $Z_{n}(K_{1, 2},W)$ in \eqref{eq:limit2} is degenerate if and only if  \eqref{eq:K12-cond1} and \eqref{eq:K12-cond2} hold. 
\end{lemma}

\begin{proof}
From \eqref{eq:K12W}  the {2-point conditional graphon induced by $K_{1, 2}$} is given by
\begin{align}
    W_{K_{1, 2}}(x,y) = \frac{1}{2}\left\{W(x,y)\left(d_{W}(x)+d_{W}(y)\right) + U_1(x, y)  \right\} . 
\end{align}
Furthermore, \eqref{eq:degreeH} yields $d_{W_{K_{1, 2}}}=\frac{6}{4}t(K_{1,2},W)=\frac{3}{2}\int_0^1  d_W(x)^2\,\ddx.$ Hence, Proposition \ref{ppn2} shows that $\Spec^-(W_{K_{1, 2}})=\emptyset$ if and only if \eqref{eq:K12-cond1} holds.

Furthermore, recalling \eqref{eq:sigma_K12} we have,
\begin{align}\label{eq:sigmaK12}
    \sigma_{K_{1, 2},W}^{2} 
=2\left[ t\left(K_{1,3},W\right) + t(P_3,W)-t\bigpar{K_{1, 3}^{+},W}-t\left(P_3^{+},W\right)\right]
\end{align}
where the graphs $K_{1,3},K_{1,3}^{+},P_3$ and $P_3^{+}$ are as shown in Figure \ref{figure:example_edge_K12}. By evaluating the densities in \eqref{eq:sigmaK12}, we obtain
\begin{align}\label{eq:varK12equiv}
    \sigma_{K_{1, 2},W}^{2} =
2\int \left\{ d_W(x) d_W(y) + d_W(x)^2 \right\} W(x,y)(1-W(x,y))\dd x
  \dd y . 
\end{align}
This shows that, $\sigma_{K_{1, 2},W}^{2} =0$ equivalent to \eqref{eq:K12-cond2}.
\end{proof}

The following theorem 
is the counterpart of 
Theorem \ref{thm:condition} for $K_{1,2}$,
and shows that for $H=K_{1,2}$, 
Theorem \ref{thm:asymp_count}(1) or (2) will give a non-degenerate limit.
By Lemma \ref{LC4}, Theorem \ref{thm:K12Z} 
is equivalent to the claim that
whenever $W$ is $K_{1,2}$-regular, 
\eqref{eq:K12-cond1} and \eqref{eq:K12-cond2} cannot occur simultaneously.
The proof of Theorem \ref{thm:K12Z} is given in Section \ref{sec:K12Zpf}. 

\begin{theorem}\label{thm:K12Z}
Suppose $W$ is a $K_{1, 2}$-regular graphon with $t(K_{1, 2},W)>0$ and $W$
is not identically $1$ a.e. 
Then,
the limit of $Z_{n}(C_{4},W)$ in \eqref{eq:limit2} is non-degenerate.
\end{theorem}

\section{Proof of Theorem \ref{thm:asymp_count}}\label{sec:proof}

Fix a graphon $W \in \mathcal{W}_0$ and a non-empty simple  graph  $H=(V(H), E(H))$ with vertices labeled $V(H)=\{1, 2, \ldots, |V(H)|\}$, and recall the definition of  $X_{n}(H,W)$ from \eqref{eq:XHW}. To express $X_{n}(H,W)$ as a generalized $U$-statistic note that 
\begin{align}
    X_{n}(H,W)=\sum_{1\leq i_{1}<\cdots<i_{|V(H)|}\leq n}f(U_{i_{1}},\cdots, U_{i_{|V(H)|}},Y_{i_{1}i_{2}},\cdots, Y_{i_{|V(H)|-1}i_{|V(H)|}})
\end{align}
where $\mathscr{G}_H := \mathscr{G}_H(\{1, 2, \ldots, |V(H)|\})$ and 
\begin{align}\label{eq:def_f}
    f(U_{1},\cdots, U_{|V(H)|},Y_{12},\cdots, Y_{|V(H)|-1~ |V(H)|})=\sum_{H'\in \mathscr{G}_H}\prod_{(a, b) \in E(H')}\one\left\{Y_{ab}\leq W(U_{a},U_{b}) \right\}. 
\end{align}
This is exactly in the framework of generalized $U$-statistics considered in  \cite{janson1991asymptotic}. Therefore, we can now orthogonally expand the function $f$  as a sum over subgraphs of the complete graph as explained in the section below.

\subsection{Orthogonal Decomposition of Generalized $U$-Statistics}
\label{sec:generalized_U}

We recall some notations and definitions from \cite{janson1991asymptotic}.
Suppose $\{U_{i}: 1 \leq i \leq n\}$ and $\{Y_{ij}: 1\leq i<j \leq n\}$ are
i.i.d.\ sequences of $U[0,1]$ random variables. Denote by $K_n$ the complete
graph on the set of vertices $\{1, 2, \ldots, n\}$ and let $G= (V(G), E(G))$
be a subgraph of $K_{n}$. Let $\mathcal{F}_{G}$ be the $\sigma$-algebra
generated by the collections $\{U_{i}\}_{i\in V(G)}$ and $\{Y_{ij}\}_{ij\in E(G)}$,
and let $L^{2}(G)=L^2(\cF_G)$ 
be the space of all square integrable random
variables that are functions of $\{U_{i}: 1 \leq i \leq n\}$ and $\{Y_{ij}:
1\leq i<j \leq n\}$. Now, consider the following subspace of $L^{2}(G)$: 
\begin{align}
M_{G}:=\{Z\in L^{2}(G) : \mathbb{E}[ZV]=0\text{ for every }V\in L^{2}(H)\text{ such that }H\subset G\}. 
\end{align}
(For the empty graph, $M_{\emptyset}$ is the space of all constants.) 
Equivalently,
$Z\in M_{G}$ if and only if $Z\in L^{2}(G)$ and 
\begin{align}\label{eq:condition_MG}
\mathbb{E}\left[Z\mid X_{i},Y_{ij}: i \in V(H), (i, j) \in E(H)\right] = 0,
\quad \text{for all } H\subset G.
\end{align}
Then,  we have the orthogonal decomposition 
\cite[Lemma 1]{janson1991asymptotic} 
\begin{align}\label{eq:LGH}
L^{2}(G)=\bigoplus_{H\subseteq G} M_H, 
\end{align}
that is, $L^{2}(G)$ is the orthogonal direct sum of $M_{H}$ for all 
subgraphs $H\subseteq G$. 
This 
allows us to decompose 
 any function in $L^{2}(G)$ as the sum of  its
projections onto $M_{H}$ for $H\subseteq G$. 
For any closed subspace $M$ of $L^2(K_n)$, denote the orthogonal projection
onto $M$ by $P_M$.
Then, in particular, for $f$ as in
\eqref{eq:def_f}, we have the  decomposition 
\begin{align}\label{eq:fH}
f=\sum_{H\subseteq G} f_{H},
\end{align}
where $f_{H}=P_{M_H}f$ is the orthogonal
projection of $f$ onto $M_{H}$. Further, for $1 \leq s \leq |V(H)|$, define
\begin{align}\label{eq:fs}
    f_{(s)}:=\sum_{H\subseteq G: |V(H)|=s} f_{H} .
\end{align}
The smallest positive $d$ such that $f_{(d)}\neq 0$ is called the {\it
  principal degree} of $f$.   The asymptotic distribution of $X_n(H, W)$
depends on the principal degree of $f$ and the geometry of the subgraphs
which appear in its decomposition.

For any graph $G\subseteq K_n$, the orthogonal projection 
onto $L^2(G)=L^2(\cF_G)$ equals the conditional expectation
$\E(\cdot\mid\cF_G)$, i.e.,
\begin{align}\label{pl1}
  P_{L^2(G)}=\E\sqpar{\cdot\mid\cF_G}.
\end{align}
Moreover, by \eqref{eq:LGH}, we have
\begin{align}\label{pl2}
P_{L^{2}(G)}=\sum_{H\subseteq G} P_{M_H}. 
\end{align}
The equations \eqref{pl1}--\eqref{pl2} enable us to express any $P_{M_H}$
as a linear combination of conditional expectations. We will do this
explicitly for the simplest cases in lemmas below.

\subsection{Proof of Theorem~\ref{thm:asymp_count}(1)} 

Recall the definition of the function $f$ from \eqref{eq:def_f} and consider its decomposition as in \eqref{eq:fH}. Then \eqref{eq:fs} for $s=1$ gives, 
\begin{align}\label{eq:fK}
    f_{(1)}=\sum_{a=1}^{|V(H)|}f_{K_{\{a\}}}, 
\end{align}
where $K_{\{a\}}$ is the graph with the single vertex $a$ and $f_{K_{\{a\}}}$ is the projection of $f$ onto the space $M_{K_{\{a\}}}$, for $1 \leq a \leq |V(H)|$. 
We will calculate $f_{\Ka}$ using the following lemma, which we state for 
general functions $F$.

\begin{lemma}\label{lemma:f_K_i}  For $1 \leq a \leq |V(H)|$, 
and any $F\in L^2$,
the projection of $F$ onto the space $M_{K_{\{a\}}}$ is given by
  \begin{align}\label{lPka}
F_{K_{\{a\}}}=\mathbb{E}\left[F\mid U_{a}\right]-\mathbb{E}[F].     
  \end{align}
\end{lemma}

\begin{proof} 
By \eqref{pl2} and \eqref{pl1},
\begin{align}
F_{\Ka}:= P_{ M_{\Ka}} F = P_{L^2(\Ka)}F-P_{M_\emptyset}F
= \E[F\mid U_a]-\E[F].
\end{align}
\end{proof}

Applying Lemma \ref{lemma:f_K_i} to  $f$ defined in \eqref{eq:def_f},  
we obtain
\begin{align}\label{fKa}
    f_{K_{\{a\}}}
    &=\sum_{H'\in\mathscr{G}_H}\mathbb{E}\left[\prod_{(b, c) \in
      E(H')}\one\left\{Y_{bc}\leq W(U_{b},U_{c})\right\}\,\middle|\,
      U_{a}\right]-\mathbb{E}[f] \notag\\
 &=\sum_{H'\in\mathscr{G}_H}\mathbb{E}\left[\prod_{(b, c)\in
   E(H')}W(U_{b},U_{c})\,
\middle|\, U_{a}\right]-\mathbb{E}[f]\notag\\
    &=\sum_{H'\in\mathscr{G}_H}t_{a}(U_{a},H',W)-\mathbb{E}[f],
\end{align}
where the last step follows from the definition of the 1-point conditional homomorphism density function (recall Definition \ref{defn:marked_density}). Then from \eqref{eq:fK}, 
\begin{align}\label{eq:fK_XHW}
    f_{(1)}=\sum_{a=1}^{|V(H)|}\left(\sum_{H'\in\mathscr{G}_H}t_{a}(U_{a},H',W)-\mathbb{E}[f]\right).
\end{align} 

We now proceed to
compute $\Var f_{(1)}$.

For this, we need the following combinatorial identity. 

\begin{lemma}\label{lemma:var_identity} For the vertex join operation $\bigoplus_{a,b}$ as in Definition \ref{defn:H1H2ab} the following holds: 
\begin{align}\label{eq:var_indentity} 
|\mathscr{G}_H|^{2}\sum_{1\leq a,b\leq |V(H)|}t\left(H\bigoplus_{a,b} H, W\right) 
&=|V(H)|^{2}\sum_{H_{1},H_{2}\in\mathscr{G}_H}t\left(H_{1}\bigoplus_{1,1}H_{2},W\right). 
\end{align}
\end{lemma} 

\begin{proof} 
For any permutation $\phi: V(H)\rightarrow V(H)$,
we define the permuted graph $\phi(H):=(\phi(V(H)), \phi(E(H)))$, where
$\phi(V(H))=\{\phi(a): 1\leq a \leq |V(H)|\}$ and $\phi(E(H))=\{(\phi(a),
\phi(b)):  (a, b) \in E(H)\}$. 

First, fix $(a,b)\in V(H)^{2}$ and consider two permutations, $\phi_{a}: V(H)\rightarrow V(H)$ and $\phi_{b}: V(H)\rightarrow V(H)$ such that $\phi_{a}(a)=\phi_{b}(b)=1$. Then 
\begin{align}\label{eq:abH1H2}
    \sum_{1\leq a,b\leq |V(H)|}\sum_{H_{1},H_{2}\in\mathscr{G}_H}t\left(H_{1}\bigoplus_{a,b}H_{2},W\right)
    &=\sum_{1\leq a,b\leq |V(H)|}\sum_{H_{1},H_{2}\in\mathscr{G}_H}t\left(\phi_{a}(H_{1})\bigoplus_{1,1}\phi_{b}(H_{2}),W\right) \nonumber \\
    &=\sum_{1\leq a,b\leq |V(H)|}\sum_{H_{1},H_{2}\in\mathscr{G}_H}t\left(H_{1}\bigoplus_{1,1}H_{2},W\right) \nonumber \\
    &=|V(H)|^{2}\sum_{H_{1},H_{2}\in\mathscr{G}_H}t\left(H_{1}\bigoplus_{1,1}H_{2},W\right), 
\end{align}
where the second equality follows, since the map $(H_1, H_2) \rightarrow (\phi_{a}(H_{1}),\phi_{b}(H_{2}))$ is a bijection from $\mathscr{G}_H^{2}$ to $\mathscr{G}_H^{2}$, for all $1 \leq a, b \leq |V(H)|$.

Next, fix $H_{1},H_{2}\in\mathscr{G}_H$. Then consider isomorphisms $\phi_{1}, \phi_{2} : V(H)\rightarrow V(H) $ such that $\phi_{1}(H_{1})=H$ and $\phi_{2}(H_{2})=H$. Thus,
\begin{align}\label{eq:H1H2GH}
    \sum_{H_{1},H_{2}\in\mathscr{G}_H}\sum_{1\leq a, b\leq |V(H)|}t\left(H_{1}\bigoplus_{a,b}H_{2},W\right)
    &=\sum_{H_{1},H_{2}\in\mathscr{G}_H}\sum_{1\leq a, b\leq |V(H)|}t\left(H\bigoplus_{\phi_{1}(a),\phi_{2}(b)}H,W\right) \nonumber \\
    &=\sum_{H_{1},H_{2}\in\mathscr{G}_H}\sum_{1\leq a, b\leq |V(H)|}t\left(H\bigoplus_{a,b}H , W \right) \nonumber \\
    &=|\mathscr{G}_H|^{2}\sum_{1\leq a, b\leq |V(H)|}t\left(H\bigoplus_{a,b}H , W\right).
\end{align}
Here,  the second equality follows since $(a, b) \rightarrow (\phi_{1}(a),\phi_{2}(b))$ is a bijection from $V(H)^2$ to $V(H)^2$. 

Combining \eqref{eq:abH1H2} and \eqref{eq:H1H2GH} the identity in \eqref{eq:var_indentity} follows. 
\end{proof}

\begin{lemma}\label{Lvarf1}
\begin{align}\label{eq:variancef1} 
    \Var[f_{(1)}]&
 =|V(H)| |\mathscr{G}_H|^{2} \left\{ \frac{1}{|V(H)|^2}
\sum_{1\leq a,b\leq |V(H)|}t\left(H\bigoplus_{a,b}H,W\right) - t(H,W)^{2}  \right\}
 \notag\\
&=
\frac{|V(H)|!\,(|V(H)|-1)!}{|\Aut(H)|^2}  \left\{ 
\sum_{1\leq a,b\leq |V(H)|}t\left(H\bigoplus_{a,b}H,W\right) 
-{|V(H)|^2} t(H,W)^{2}  \right\}
.\end{align}
\end{lemma}

\begin{proof}
Recalling \eqref{eq:fK_XHW} gives, 
since the terms in the outer sum there are independent,
\begin{align}\label{eq:varf1}
    \text{Var}[f_{(1)}]=\sum_{a=1}^{|V(H)|}\text{Var}\left[\sum_{H'\in\mathscr{G}_H}t_{a}(U_{a},H',W)\right].
\end{align}
Consider the term corresponding to $a=1$ in the sum above. 
 For any $H_{1},H_{2}\in\mathscr{G}_H$,
\begin{align}
    \mathbb{E}\left[t_{1}(U_{1},H_{1},W)t_{1}(U_{1},H_{2},W)\right] 
    &=t\left(H_{1}\bigoplus_{1,1}H_{2},W\right) . 
\end{align}
Hence,
\begin{align}\label{eq:varf11}
    \text{Var}\left[\sum_{H'\in\mathscr{G}_H}t_{1}(U_{1},H',W)\right]
    &=\sum_{H_{1},H_{2}\in\mathscr{G}_H}\text{Cov}\left[t_{1}(U_{1},H_{1},W),t_{1}(U_{1},H_{2},W)\right] \nonumber \\
    &=\sum_{H_{1},H_{2}\in\mathscr{G}_H}\left(t\left(H_{1}\bigoplus_{1,1}H_{2},W\right)-t(H,W)^{2}\right) . 
\end{align}
Now, an argument similar to Lemma \ref{lemma:var_identity} shows that 
\begin{align}
    \sum_{H'\in\mathscr{G}_H}t_{a}(x,H',W)=\sum_{H'\in\mathscr{G}_H}t_{b}(x,H',W) , 
    \end{align}
for all $x\in [0,1]$ and $1\leq a, b\leq |V(H)|$. 
Hence, \eqref{eq:varf1} and \eqref{eq:varf11} imply
\begin{align}\label{eq:variancef1a} 
    \text{Var}[f_{(1)}]&= |V(H)|
                         \sum_{H_{1},H_{2}\in\mathscr{G}_H}\left(t\left(H_{1}\bigoplus_{1,1}H_{2},W\right)-t(H,W)^{2}\right),
\end{align}
and the result follows by Lemma \ref{lemma:var_identity},
using \eqref{eq:fact_Aut} for the second equality.
\end{proof}

Note that $\E f_{(1)}= 0$ by \eqref{eq:fs}.
Hence $\Var \fett=0 $ if and only if $\fett=0$ a.s.

\begin{lemma}\label{lm:nregularf_WH}
$\Var\fett=0$ if and only if\/ $W$ is  $H$-regular.
\end{lemma}

\begin{proof}
Lemma \ref{Lvarf1} shows that $\text{Var}[f_{(1)}]$ is zero if and only if 
\begin{align}\label{eq:reg_var}
    \frac{1}{|V(H)|^{2}}\sum_{1\leq a,b\leq |V(H)|}t\left(H\bigoplus_{a,b},W\right)=t(H,W)^{2} . 
\end{align}
Now observe,
\begin{align}
    \sum_{1\leq a,b\leq |V(H)|}t\left(H\bigoplus_{a,b}H,W\right)
    &=\sum_{1\leq a,b\leq |V(H)|}\int t_{a}(x,H,W)t_{b}(x,H,W)\dd x \notag\\
    &=\int \left(\sum_{1\leq a\leq |V(H)|}t_{a}(x,H,W)\right)^{2}\dd x .
\end{align}
Thus \eqref{eq:reg_var} becomes,
using also \eqref{intta},
\begin{align}\label{eq:reg_var_eq}
    \int \left(\sum_{1\leq a\leq |V(H)|}t_{a}(x,H,W)\right)^{2}\dd x-\left(
  \int \sum_{1\leq a\leq |V(H)|}t_{a}(x,H,W)\right)^{2} \dd x = 0 
,\end{align}
which is equivalent to
$\text{Var}\left[\Delta(U)\right]=0$, 
where we define 
\begin{align}
\Delta(x):=\sum_{1\leq a\leq |V(H)|}t_{a}(x,H,W)  
\end{align}
and let
$U\sim\operatorname{Uniform}[0,1]$. 
Hence, $\text{Var}[f_{(1)}]=0$ if and only if
$\Delta(U)$ is constant a.s. Therefore, since $\mathbb{E}\Delta(U)=
|V(H)| t(H,W)$, we see that  $\text{Var}[f_{(1)}]=0$ if and only if  
\begin{align}\label{regH}
     \frac{1}{|V(H)|}\sum_{1\leq a\leq |V(H)|}t_{a}(x,H,W)=t(H,W) \text{ for almost every } x \in [0, 1]. 
    \end{align}
By Definition \ref{defn:regular},  
\eqref{regH} says that $W$ is $H$-regular. 
\end{proof}

\begin{proof}[Proof of Theorem~\ref{thm:asymp_count}(1)]
Lemma \ref{lm:nregularf_WH} shows that 
if $W$ is not $H$-regular, then
the principal degree of $f$ is 1. 
Thus, \cite[Theorem 1]{janson1991asymptotic} yields
   \begin{align}\label{lvb1}
        \frac{X_{n}(H,W)-\frac{(n)_{|V(H)|}}{|\Aut(H)|}t(H,W)}{n^{|V(H)|-\frac{1}{2}}}  \overset{D}{\rightarrow}\mathsf{N}(0, \tau^{2}),
    \end{align}
    where,
using also \eqref{eq:variancef1} and \eqref{ttau},
    \begin{align}\label{lvb2}
    \tau^{2} & = \frac{1}{|V(H)|!\,(|V(H)|-1)!}\text{Var}[f_{(1)}] 
= \tau^2_{H, W}
.\end{align}
This completes the proof of Theorem~\ref{thm:asymp_count}(1)
when $W$ is not $H$-regular.

In fact, \eqref{lvb1}--\eqref{lvb2} hold also when $W$ is $H$-regular, with
$\fett=0$ and $\tau^2=0$. Although this case is not included in the
statement of \cite[Theorem 1]{janson1991asymptotic}, it follows by its proof, as
a consequence of \cite[Lemma 2]{janson1991asymptotic}; see also 
\cite[Corollary 11.36]{SJIII}.
Consequently,  Theorem~\ref{thm:asymp_count}(1) holds for any 
$W \in \mathcal{W}_0$. 
\end{proof}

\subsection{Proof of Theorem~\ref{thm:asymp_count}(2)} 

In this case, $W$ is $H$-regular, hence $f_{(1)}\equiv 0$ by Lemma
\ref{lm:nregularf_WH}.
Therefore, we consider $f_{(2)}$ (recall \eqref{eq:fs}) which can be written as 
\begin{align}
 f_{(2)}=\sum_{1\leq a< b\leq |V(H)|} \left( f_{\Kab} +f_{\Qab} \right) , 
\end{align}
where $\Kab = (\{a, b\}, \emptyset)$ is the graph with two vertices
$a$ and $b$ and no edges, and $\Qab = (\{a, b\}, \{ (a, b) \} )$ is
the complete graph with vertices $a$ and $b$. 
As for $\fett$, we have $\E\ftva=0$, and thus $\Var \ftva=0\iff \ftva=0$ a.s.

If $\Var\ftva\neq0$, then $f$ has principal degree 2, and we can apply
\cite[Theorem 2]{janson1991asymptotic},
which shows that
   \begin{align}\label{beet1}
        \frac{X_{n}(H,W)-\frac{(n)_{|V(H)|}}{|\Aut(H)|}t(H,W)}{n^{|V(H)|-{1}}}
     \overset{D}{\rightarrow}
\gs Z+\sum_{\gl\in\gL}\gl(Z_\gl^2-1),
    \end{align}
    where
$Z$ and $\set{Z_\gl}_{\gl\in\gL}$ are independent standard Gaussians,
\begin{align}\label{beet2}
  \gss=\frac{1}{2(|V(H)-2)!^2}\E \bigsqpar{f_{\Qij}^2}
\end{align}
and $\gL$ is the multiset of (non-zero) eigenvalues of
a certain integral operator $T$.

Moreover, if $\Var\ftva=0$, so $\ftva=0$ a.s., then the conclusion of
\cite[Theorem 2]{janson1991asymptotic} still holds (with a trivial limit 0),
again as a consequence of 
\cite[Lemma 2]{janson1991asymptotic}.
(See also the more general \cite[Theorem 11.35]{SJIII}.)
Hence, \eqref{beet1} holds in any case.

It remains to show that $\gss=\gss_{H,W}$ in \eqref{gss}, and that $\gL$
equals
$\Spec^-(W_H)$; then \eqref{beet1} yields \eqref{tt2}.
We begin by finding $f_\Kab$ and $f_\Qab$.
 
\begin{lemma}\label{lemma:Kab_projection} 
For $1\leq a< b\leq |V(H)|$ and any $F\in L^2$,
the projection of $f$ onto the space $M_{\Kab}$ is given by
\begin{align}\label{lPkab}
    F_{\Kab}=\mathbb{E}\left[F\mid U_{a},U_{b}\right]-\mathbb{E}[F\mid U_{a}]-\mathbb{E}[F\mid U_{b}]+\mathbb{E}[F] .  
\end{align}
\end{lemma} 

\begin{proof}
By \eqref{pl2}, 
\begin{align}
F_{\Kab} := 
P_{ M_{\Kab}} F &= 
P_{L^2(\Kab)}F - P_{M_\Ka}F - P_{M_{\Kb}}F -P_{M_\emptyset}F
\notag\\&
 = P_{L^2(\Kab)}F - P_{L^2(\Ka)}F - P_{L^2(\Kb)}F +P_{M_\emptyset}F
\end{align}
and the result follows by \eqref{pl1}.
\end{proof}

\begin{lemma}\label{lemma:Qab_projection} 
For $1\leq a< b\leq |V(H)|$ and any $F\in L^2$,
the projection of $f$ onto the space $M_{\Qab}$ is given by
\begin{align}\label{lPqab}
    F_{\Qab}=\mathbb{E}\left[F\mid U_{a},U_{b},Y_{ab}\right]
- \mathbb{E}\left[F\mid U_{a},U_{b}\right].
\end{align}
\end{lemma} 

\begin{proof}
  The subgraphs of $\Qab$ are $\Kab$, $\Ka$, $\Kb$ and $\emptyset$, and thus
\eqref{pl2} yields
\begin{align}
F_{\Qab} := 
P_{ M_{\Qab}} F& = 
P_{L^2(\Qab)}F -P_{M_\Kab}F- P_{M_\Ka}F - P_{M_{\Kb}}F -P_{M_\emptyset}F
\notag\\&
 = P_{L^2(\Qab)}F - P_{L^2(\Kab)}F
,\end{align}
and the result follows by \eqref{pl1}.
\end{proof}

Specializing to  $f$ defined in \eqref{eq:def_f},  
we  found $f_\Ka=\E \sqpar{f\mid U_a}-\E f$ in \eqref{fKa}.
Furthermore, the same argument yields, recalling \eqref{ta} and \eqref{tab-},
\begin{align}\label{efuu}
  \E\sqpar{f\mid U_a,U_b}=
\sum_{H'\in\mathscr{G}_H}t_{a,b}(U_{a},U_b,H',W)
\end{align}
and  
\begin{align}\label{efuuy}
    \E\sqpar{f\mid U_a,U_b,Y_{ab}}
&= \sum_{H'\in\mathscr{G}_H}t^-_{a,b}(U_{a},U_b,H',W)  Z_{H',\setab}(Y_{ab},U_{a},U_{b}),
\end{align}
where
\begin{align}\label{ZH'}
    Z_{H',\setab}(Y_{ab},U_{a},U_{b})
    :=
    \begin{cases}
        \one\{Y_{ab}\leq W(U_{a},U_{b})\} & \text{ if } (a, b) \in E(H'),\\
1 & \text{ otherwise. } 
    \end{cases}
\end{align}
Let also
\begin{align}
    \overline W_{H',  \{a, b\}}(x, y) :=
    \begin{cases}
    W(x, y) & \text{ if } (a, b)\in E(H'),\\
    1 & \text{ otherwise},
    \end{cases}
\end{align} 
and
$\mathscr{G}_{H, \{a, b\}}:=\{H'\in \mathscr{G}_H: (a, b) \in E(H')\}$. 
Then, \eqref{lPqab}, \eqref{efuu} and \eqref{efuuy} yield,
using also \eqref{tab-tab},
\begin{align}\label{rom}
  f_\Qab &=
\sum_{H'\in\mathscr{G}_H}t^-_{a,b}(U_{a},U_b,H',W) 
\Bigpar{ Z_{H',\setab}(Y_{ab},U_{a},U_{b}) - \overline W_{H',  \{a, b\}}(U_a, U_b) }
\notag\\
&=
\sum_{H'\in\mathscr{G}_{H,\setab}} t^-_{a,b}(U_{a},U_b,H',W)  
\Bigpar{\one\{Y_{ab}\leq W(U_{a},U_{b})\} -W(U_a,U_b)}.
\end{align}

To compute the variance of $f_{\Qij}$, we 
 recall the notions of weak and strong edge joins from
Definition~\ref{defn:H1H2ab},
and
introduce a few
definitions. 
Let $V_H^2=\{(a,b)\in V(H)^{2}: a\neq b\}$. For $(a,b),(c,d)\in V_H^2$ define
\begin{align}\label{def:tilde_minus}
    \underline{t}\left(H_{1}\bigominus_{(a,b),(c,d)} H_{2},W\right) =  
    t\left(H_{1}\bigominus_{(a,b),(c,d)} H_{2},W\right) \bm 1\{(a,b) \in E^{+}(H_{1}) \text{ and }(c,d) \in E^{+}(H_{2})\}
 \end{align} 
and similarly, 
\begin{align}\label{def:tilde_plus}
\underline{t}\left(H_{1}\bigoplus_{(a,b),(c,d)} H_{2},W\right) = 
t\left(H_{1}\bigoplus_{(a,b),(c,d)} H_{2},W \right) \bm 1\{(a,b) \in E^{+}(H_{1}) \text{ and }(c,d) \in E^{+}(H_{2})\} .
\end{align}
Then we have the following identities, similar to Lemma
\ref{lemma:var_identity}: 
\begin{lemma}\label{lemma:edge_join}
Let $V_H^2$ be as defined above,
and let $K_H:=|V_H^2|^2=|V(H)|^{2}(|V(H)|-1)^{2}$. 
Then 
\begin{align}\label{eq:wjoin_indentity}
 K_H \sum_{H_{1},H_{2}\in\mathscr{G}_H} \underline{t}
  \left(H_{1}\bigominus_{(1,2),(1,2)} H_{2},W\right) 
& = |\mathscr{G}_H|^{2}\sum_{(a,b),(c,d)\in V_H^2 } \underline{t} \left(H\bigominus_{(a,b),(c,d)}  H,W\right). 
\end{align} 
and, similarly,
\begin{align}\label{eq:sjoin_identity}
 K_H \sum_{H_{1},H_{2}\in\mathscr{G}_H} \underline{t}
  \left(H_{1}\bigoplus_{(1,2),(1,2)} H_{2},W\right) 
& = |\mathscr{G}_H|^{2}\sum_{(a,b),(c,d)\in V_H^2 } \underline{t} \left(H\bigoplus_{(a,b),(c,d)}  H,W\right). 
\end{align} 
\end{lemma}

\begin{proof} We will first show that 
\begin{align}\label{eq:wjoin_ab}
  \sum_{(a,b),(c,d)\in V_H^2}\sum_{H_{1},H_{2}\in\mathscr{G}_H} \underline{t} \left(H_{1}\bigominus_{(a,b),(c,d)} H_{2},W\right) &= K_H \sum_{H_{1},H_{2}\in\mathscr{G}_H} \underline{t} \left(H_{1}\bigominus_{(1,2),(1,2)} H_{2},W\right) 
.\end{align} 
For this consider permutations $\phi_{(a,b)}, \phi_{(c,d)} : V(H) \rightarrow V(H)$ such that $\phi_{(a,b)}(a)=1$ and $\phi_{(a,b)}(b)=2$, and $\phi_{(c,d)}(c)=1$ and $\phi_{(c, d)}(d)=2$. Then
\begin{align}
    \sum_{(a,b),(c,d)\in V_H^2}\sum_{H_{1},H_{2}\in\mathscr{G}_H} \underline{t} \left(H_{1}\bigominus_{(a,b),(c,d)} H_{2},W\right)
    &=\sum_{V_H^2\times V_{H}^2}\sum_{\mathscr{G}_H^{2}} \underline{t}
      \left(\phi_{(a,b)}(H_{1})\bigominus_{(1,2),(1,2)} \phi_{(c,d)}(H_{2}),W\right)
\notag\\
    &= K_H \sum_{\mathscr{G}_H^{2}} \underline{t} \left(H_{1}\bigominus_{(1,2),(1,2)} H_{2},W\right),
\end{align}
where the last equality follows from the observation that $(H_{1},H_{2}) \rightarrow (\phi_{(a,b)}(H_{1}),\phi_{(c,d)}(H_{2}))$ is an bijection from $\mathscr{G}_H^2 $ to $\mathscr{G}_H^2$, for all $(a,b),(c,d)\in V_H^2$.

Now by considering isomorphisms $\phi_{1}$ and $\phi_{2}$ such that $\phi_{1}(H_{1})=H$ and $\phi_{2}(H_{2})=H$, a similar argument as above shows that 
\begin{align}\label{eq:wjoin_H12}
    \sum_{(a,b),(c,d)\in V_H^2}\sum_{H_{1},H_{2}\in\mathscr{G}_H} \underline{t}\left(H_{1}\bigominus_{(a,b),(c,d)} H_{2},W\right)
    &=|\mathscr{G}_H|^{2}\sum_{(a,b),(c,d)\in V_H^2 } \underline{t} \left(H\bigominus_{(a,b),(c,d)}  H,W\right).
\end{align} 

Combining \eqref{eq:wjoin_ab} and \eqref{eq:wjoin_H12} yields the identity
\eqref{eq:wjoin_indentity}. The identity \eqref{eq:sjoin_identity} follows
by the same proof with only notational differences.
\end{proof}

With the above definitions and identities we now proceed to compute the variance of $f_{\Qij}$. 

\begin{lemma}\label{lm:Qab_variance}
We have
\begin{align}\label{qabvar}
   \Var[f_{\Qij}]=\frac{(|V(H)|-2)!^{2}}{|\Aut(H)|^{2}}\sum_{(a,b),(c,d)\in E^{+}(H)}\lrpar{t\left(H\bigominus_{(a,b),(c,d)} H,W\right)-t\left(H\bigoplus_{(a,b),(c,d)} H,W\right)}
.\end{align}
\end{lemma}
\begin{proof}
We specialize \eqref{rom} to $(a,b)=(1,2)$, and write for convenience
\begin{align}
 h(U_{1},U_{2}, H_{1}, H_{2},W)
:=\tij^-(U_{1},U_{2},H_{1},W)\, \tij^-(U_{1},U_{2},H_{2},W).
\end{align} 
This yields, 
\begin{align}\label{eq:varianceQ12}
    \mathbb{E}[f_{\Qij}^2]
    & = \sum_{H_{1},H_{2}\in\mathscr{G}_{H, \{1, 2\}}}\mathbb{E} \left[ h(U_{1},U_{2}, H_{1}, H_{2},W) \bigpar{\one\{Y_{12}\leq W(U_{1},U_{2})\} - W(U_1, U_2) }^2 \right] \nonumber \\
    & = \sum_{H_{1},H_{2}\in\mathscr{G}_{H, \{1, 2\}}}\mathbb{E}\bigsqpar{ h(U_{1},U_{2},H_{1}, H_{2}, W) W(U_{1},U_{2}) (1-W(U_{1},U_{2}))} \nonumber \\
& =\sum_{H_{1},H_{2}\in\mathscr{G}_{H, \{1, 2\}}}
\lrpar{
t\left(H_{1}\bigominus_{(1,2),(1,2)} H_{2},W\right)
-t\left(H_{1}\bigoplus_{(1,2),(1,2)} H_{2},W\right) }.  
\end{align}
Now, using the notations introduced in \eqref{def:tilde_minus} and \eqref{def:tilde_plus}, the identity \eqref{eq:varianceQ12} can be written as 
\begin{align}
    \mathbb{E}[f_{\Qij}^2] & 
= \sum_{H_{1},H_{2} \in \mathscr{G}_H}\lrpar{\underline{t}\left(H_{1}\bigominus_{(1,2),(1,2)} H_{2},W\right)-\underline{t}\left(H_{1}\bigoplus_{(1,2),(1,2)} H_{2},W\right)}  
\notag\\
    & = \frac{(|V(H)|-2)!^2}{|\Aut(H)|^2}\sum_{ (a,b),(c,d)\in V_H^2 }\lrpar{ \underline{t}\left(H\bigominus_{(a,b),(c,d)} H,W\right) - \underline{t}\left(H\bigoplus_{(a,b),(c,d)} H,W\right) }
\notag\\
    & = \frac{(|V(H)|-2)!^2}{|\Aut(H)|^2}\sum_{(a,b),(c,d)\in E^{+}(H)}\lrpar{t\left(H\bigominus_{(a,b),(c,d)} H,W\right)-t\left(H\bigoplus_{(a,b),(c,d)} H,W\right)}, 
\end{align} 
where the second equality uses the identities from Lemma \ref{lemma:edge_join}
and \eqref{eq:fact_Aut},
and the third equality follows from the definitions in
\eqref{def:tilde_minus} and \eqref{def:tilde_plus}.  
This yields the result \eqref{qabvar}, since $\E f_\Qab=0$.
\end{proof}

Lemma \ref{lm:Qab_variance} and \eqref{beet2} show that
\begin{align}\label{gss=}
  \gss=\gss_{H,W},
\end{align}
as defined in \eqref{gss}.

Next, we compute the Hilbert--Schmidt operator $T$ as defined in
\cite[Theorem 2]{janson1991asymptotic}. 
Note first that in our case this operator is defined on the space
$M_{K_{\{1\}}}$. 
Recall that $M_\Ki\subset L^2(\Ki)$, where $L^2(\Ki)$ is the space of all
square integrable random variables of the form $g(U_1)$.
We may identify $L^2(\Ki)$ and $L^2\oi$, and then 
\eqref{eq:LGH} yields the orthogonal decomposition
\begin{align}\label{eq:decomp}
    L^{2}[0,1]=M_\Ki \bigoplus M_{\emptyset}, 
\end{align}
where $M_{\emptyset}$ is the one-dimensional space of all constants. 
Hence, $M_\Ki$ is identified with the subspace of $L^2\oi$ orthogonal to
constants, i.e., $M_\Ki=\bigset{g\in L^2\oi:\int_0^1 g=0}$.

Then, taking $g,h\in M_{K_{\{1\}}}\subset L^2\oi$, the
definitions given in \cite[Theorem 2]{janson1991asymptotic} yield
\begin{align}\label{eq:Tgh}
\langle Tg,h\rangle &=
 \frac{1}{2(|V(H)|-2)!}\mathbb{E}\left[fg(U_{1})h(U_{2})\right] 
. \end{align}  
 
Recall the operator $T_{W_H}$ defined on $L^2\oi$ by \eqref{eq:TW} and
\eqref{eq:WH}. 

\begin{lemma}\label{lemma:operator_TWH} 
If $W$ is $H$-regular, then 
the operator $T$ on $M_\Ki$ defined in \eqref{eq:Tgh} 
equals the operator $T_{W_H}$ restricted to the space
$M_{K_{\{1\}}}$. Moreover, then the multiset of
non-zero
eigenvalues of\/ $T$ is equal to $\mathrm{Spec}^{-}(W_{H})$.
\end{lemma}

\begin{proof} 

We may replace $f$ by $\E\sqpar{f\mid U_1,U_2}$ in \eqref{eq:Tgh}, which by
\eqref{efuu} yields
    \begin{align}\label{Tgh2}
\langle Tg,h\rangle&    
= \frac{1}{2(|V(H)|-2)!}\mathbb{E}\bigsqpar{\E\sqpar{f\mid U_1,U_2}g(U_{1})h(U_{2})}
\notag\\
    &=\frac{1}{2(|V(H)|-2)!}\mathbb{E}\left[\sum_{H'\in \mathscr{G}_H}\tij(U_{1},U_{2},H',W)g(U_{1})h(U_{2})\right]\nonumber\\
    &=\left\langle\frac{1}{2(|V(H)|-2)!}\int\sum_{H'\in \mathscr{G}_H}\tij(x,\cdot,H',W)g(x)\mathrm{d}x,h(\cdot)\right\rangle .
\end{align}

Denote by $S_{|V(H)|}$ the set of all $|V(H)|!$ permutations of $V(H)$. 
Then it is easy to observe that 
\begin{align}\label{eq:all_iso1}
    \sum_{\phi\in S_{|V(H)|}}\tij(x,y,\phi(H),W)=|\Aut(H)|\sum_{H'\in\mathscr{G}_H}\tij(x,y,H',W)
.\end{align}
Also,
\begin{align}
    \sum_{\phi\in S_{|V(H)|}}\tij(x,y,\phi(H),W)
    &=\sum_{1\leq a\neq b\leq |V(H)|}\sum_{\substack{\phi\in S_{|V(H)|}\\\phi(a)=1,\phi(b)=2}}\tij(x,y,\phi(H),W)\nonumber\\
    &=\sum_{1\leq a\neq b\leq |V(H)|}\sum_{\substack{\phi\in S_{|V(H)|}\\\phi(a)=1,\phi(b)=2}}t_{\phi^{-1}(1)\phi^{-1}(2)}(x,y,H,W)\nonumber\\
    &=\sum_{1\leq a\neq b\leq |V(H)|}\sum_{\substack{\phi\in S_{|V(H)|}\\\phi(a)=1,\phi(b)=2}}t_{a, b}(x,y,H,W)\nonumber\\
    &=(|V(H)|-2)!\sum_{1\leq a\neq b\leq |V(H)|}t_{a,b}(x,y,H,W)\label{eq:all_iso2}
\end{align}
Combining \eqref{eq:all_iso1} and \eqref{eq:all_iso2}, we have,
recalling \eqref{eq:WH},
\begin{align}\label{eq:WH_alt}
    \frac{1}{2(|V(H)|-2)!}\sum_{H'\in\mathscr{G}_H}\tij(x,y,H',W)=\frac{1}{2|\Aut(H)|}\sum_{1\leq a\neq b\leq |V(H)|}t_{a, b}(x,y,H,W)=W_{H}(x,y).
\end{align}
Consequently, combining \eqref{Tgh2}, \eqref{eq:WH_alt} and \eqref{eq:TW},
we obtain
\begin{align}\label{TT}
  \innprod{Tg,h}=\innprod{T_{W_H} g,h},
\qquad
g,h\in M_\Ki.
\end{align}
Furthermore, since $W$ is $H$-regular, $W_H$ is degree regular and
\eqref{eq:degreeH} shows that 
\begin{align}
  \label{TWH1}
T_{W_H}1=d_{W_H}=d_{W_H}\cdot 1.
\end{align}
Hence, $T_{W_H}$ maps the space $M_\emptyset$ of constant functions into
itself. 
By \eqref{eq:decomp}, $M_\Ki$ is the  orthogonal complement of $M_\emptyset$,
and thus, since $T_{W_H}$ is a symmetric operator, 
$T_{W_H}$ also maps $M_\Ki$ into itself.
Hence both $T$ and $T_{W_H}$ map $M_\Ki$ into itself, and thus
\eqref{TT} shows that $T=T_{W_H}$ on $M_\Ki$.

Finally, recall that $\gL$ in \eqref{beet1} is the multiset of non-zero
eigenvalues of $T$, which we just have shown equals the multiset of
eigenvalues of $T_{W_H}$ on $M_\Ki$. Moreover, on $M_\emptyset$, 
$T_{W_H}$ has the single eigenvalue $d_{W_H}$ by \eqref{TWH1}. 
Hence, $\Spec(W_H)=\gL\cup\set{d_H}$,
and thus $\Spec^-(W_H)=\gL$ by the definition after \eqref{eq:degreeH}.
\end{proof}

\begin{proof}[Proof of Theorem~\ref{thm:asymp_count}(2)]
The result now follows by \eqref{beet1}, \eqref{gss=}, and
Lemma \ref{lemma:operator_TWH}.
\end{proof}

\subsection{Higher Order Limits}
\label{sec:degeneratelimit}

In the case where the limit in Theorem~\ref{thm:asymp_count}(2) is degenerate (as in Example \ref{example:3star}), the function $f$ in \eqref{eq:def_f} has principal degree $d > 2$. In this case, \cite[Theorem 3]{janson1991asymptotic} shows that
$(X_n(H, W) - \E X_n(H, W))/n^{|V(H)| - d/2}$  has a (non-degenerate) limit distribution, which can be expressed as a polynomial of degree $d$ in
(possibly infinitely many) independent standard Gaussian variables.
The expression in  \cite[Theorem 3]{janson1991asymptotic} uses Wick products
of Gaussian variables; these can be expressed using Hermite polynomials, see
\cite[Theorems 3.19 and 3.21]{SJIII}. One simple illustration (with $d=4$)
is given in Example \ref{example:3star}. This leads to the following natural  open questions: 

\begin{problem}
  For which graphs $H$ can such higher order limits (i.e., with $d\ge3$) occur?
\end{problem}

\begin{problem}
  Is it possible to have arbitrarily high order principal degree $d$?
\end{problem}

\section{Proof of Theorem \ref{thm:bizero}}
\label{sec:bipartitepf}

It is obvious from \eqref{eq:sigma-sq-exp} that if $W$ is random free, then $\sigma_{H,W}^2=0$. For the converse, suppose that $W$ is not random free. 
Then the set $P:=\bigset{(x,y)\in[0,1]^2:0<W(x,y)<1}$
has  
$|P|>0$, 
where $|\cdot|$ denotes the Lebesgue measure.
Let $(x_{0},y_{0})$ be a Lebesgue
point of $P$. Then we can find intervals $I$ and $J$ containing $x_{0}$ and
$y_{0}$ respectively such that $|P\bigcap(I\times J)|>(1-\varepsilon)
|I\times J|$ ($\varepsilon>0$ to be chosen later). Define, 
\begin{align}\label{defI'}
    P_{x}:=\left\{y\in J:(x,y)\in P\right\} \quad \text{ and } \quad I':=\left\{x\in I:|P_{x}|>(1-\delta)|J|\right\}, 
\end{align}
where $\delta>0$ will be chosen later. Then,
\begin{align}
    \delta|J|\left|I\setminus I'\right| \leq \int_{I\setminus I'}\left|J\setminus P_{x}\right|\ddx  \leq \int_{I}\left|J\setminus P_{x}\right| \ddx 
    &=\int_{I}|J| \ddx-\int_{I}|P_{x}| \ddx 
\notag\\& 
=|I||J|-\int_{I}\int_{P_x} \dd z  \ddx
\notag\\& 
    =|I||J|- |P \bigcap\left(I\times J\right) |
\notag\\& 
    <\varepsilon\left|I\times J\right|=\varepsilon|I||J| . 
\end{align}
This implies, 
\begin{align}\label{eq:I-I'}
    \left|I\setminus I'\right|\leq \frac{\varepsilon}{\delta}|I| . 
\end{align}
Similarly, defining $P^{y}:=\left\{x\in I:(x,y)\in P\right\}$ and $J':=\left\{y\in J:\left|P^{y}\right|>(1-\delta)|I|\right\}$ we have,
\begin{align}\label{eq:J-J'}
    \left|J\setminus J'\right|\leq \frac{\varepsilon}{\delta}|J| . 
\end{align}

Next, fix $a<b\in V(H)$ such that $(a,b)\in E(H)$. Suppose $H$ has bipartition $(A, B)$ and without loss of generality consider $a\in A$ and $b\in B$. Then from \eqref{eq:sigma-sq-exp} it follows that, 
\begin{align}\label{eq:lower-bdd-sigma}
    \sigma_{H,W}^2 & \geq c_H 
\int_{[0, 1]^2} t_{a, b}^{-}(x, y, H, W)^2 W(x,y)(1-W(x,y)) \ddx \ddy . 
\end{align} 
Define,
\begin{align}
    \mathcal{S}:=\bigg\{\bm{z}_{-(a,b)}:=(z_{1},  \cdots,z_{a-1},z_{a+1}, & \cdots,z_{b-1},z_{b+1},  \cdots,z_{|V(H)|}) \nonumber \\ 
    & : z_{v}\in I \text{ if } v \in A \backslash\{ a \} 
\text{ and } z_{v}\in J\text{ if } v\in B\backslash\{b\}   \bigg\} . 
\end{align}
and  
\begin{align}\label{eq:tabzxy}
t_{a, b}^{-}(\bm z_{-(a, b)}, x, y, H, W)  
=   \prod_{r \in N_H(a)\backslash \{b\}}W(x, z_r) 
\prod_{s \in N_H(b) \backslash \{a\}} W(y, z_s) 
\prod_{(r, s) \in E(H\backslash \{a, b\}) } W(z_r, z_s) .   
\end{align}
Note that 
\begin{align}\label{eq:tabxy}
\int_{[0, 1]^{|V(H)|-2}}t_{a, b}^{-}(\bm z_{-(a, b)}, x, y, H, W) \prod_{r \notin \{a, b\}} \dd z_r = t_{a, b}^{-}(x, y, H, W) .
\end{align} 
It is easy to see that $|\mathcal{S}|=|I|^{|A|-1}|J|^{|B|-1}$. Now, fix $(x,y)\in I'\times J'$. Then 
\begin{align}\label{eq:zero-bound}
\mathcal Q_0 : =    \left|\left\{\bm{z}_{-(a,b)}\in \mathcal{S}: t_{a, b}^{-}(\bm z_{-(a, b)}, x, y, H, W)  =0\right\}\right| & \leq T_1 + T_2 + T_3 , 
\end{align}
where 
\begin{align}
T_1 & := \sum_{r \in N_H(a)\backslash \{b\}} \left|\left\{\bm{z}_{-(a,b)}\in\mathcal{S}: W(x, z_r) =0\right\}\right|,  \\ 
T_2 & := \sum_{s \in N_H(b)\backslash \{a\}} \left|\left\{\bm{z}_{-(a,b)}\in\mathcal{S}: W(y, z_s) =0\right\}\right|, \\ 
T_3 & := \sum_{ (r, s) \in E(H\backslash \{a, b\})  } \left|\left\{\bm{z}_{-(a,b)}\in\mathcal{S}: W(z_r, z_s) =0\right\}\right| . 
\end{align} 
Let us now look at each term separately. We begin with $T_1$. Note that for $r \in N_H(a)\backslash \{b\}$, 
\begin{align}
    \left|\left\{\bm{z}_{-(a,b)}\in\mathcal{S}:W(x,z_{r})=0\right\}\right|
    &= \left|\left\{z_{r}\in J:W(x, z_{r})=0 \right\}\right||I|^{|A|-1}|J|^{|B|-2}
\notag\\&     
\leq \left|J\setminus P_{x}\right||I|^{|A|-1}|J|^{|B|-2}
\notag\\& 
     < \delta|I|^{|A|-1}|J|^{|B|-1}
\end{align}
where the last inequality follows from 
our assumption $x\in I'$ and \eqref{defI'}. 
This implies, 
\begin{align}\label{eq:T1}
T_1 < (d_a -1 ) \delta|I|^{|A|-1}|J|^{|B|-1} ,
\end{align}
where $d_a$ is the degree of the vertex $a$ in $H$. Similarly, 
\begin{align}\label{eq:T2}
T_2 < (d_b -1 ) \delta|I|^{|A|-1}|J|^{|B|-1} .
\end{align}
Finally, consider $T_3$. Suppose $(r, s) \in E(H\backslash \{a, b\})$  and assume without loss of generality $r \in A$ and $s\in B$.  Then,
\begin{align}
    \left|\left\{\bm{z}_{-(a,b)}\in\mathcal{S}:W(z_{r},z_{s})=0\right\}\right|
    &= \left|\left\{z_{r}\in I, z_{s}\in J :
      W(z_{r},z_{s})=0\right\}\right||I|^{|A|-2}|J|^{|B|-2}
\notag\\& 
    \leq \left|\left(I\times J\right)\setminus\left( P \bigcap\left(I\times
  J\right)\right)\right||I|^{|A|-2}|J|^{|B|-2}
\notag\\& 
    <\varepsilon|I\times J||I|^{|A|-2}|J|^{|B|-2}
\notag\\&     
   =\varepsilon|I|^{|A|-1}|J|^{|B|-1} .
\end{align} 
This implies, 
\begin{align}\label{eq:T3}
T_3 \leq (|E(H)|-d_{a}-d_{b}+1) \varepsilon|I|^{|A|-1}|J|^{|B|-1} . 
\end{align}
Combining \eqref{eq:T1}, \eqref{eq:T2}, and \eqref{eq:T3} with \eqref{eq:zero-bound} gives, 
\begin{align}
   \mathcal Q_0    &\leq \left[\left(d_{a}+d_{b}-2\right)\delta + \left(E(H)-d_{a}-d_{b}+1\right)\varepsilon\right]|I|^{|A|-1}|J|^{|B|-1}\notag\\
    &< 2 \left|E(H)\right|(\delta+\varepsilon)|I|^{|A|-1}|J|^{|B|-1} . 
\end{align}
Choosing $\delta=10 \varepsilon$ and $\varepsilon<\frac{1}{100 |E(H)|}$ gives $
\left|E(H)\right|(\delta+\varepsilon)|I|^{|A|-1}|J|^{|B|-1}<|I|^{|A|-1}|J|^{|B|-1}.$ 
Thus, 
\begin{align}
 \mathcal Q_0 < |I|^{|A|-1}|J|^{|B|-1} , 
\end{align}
and hence, $|\mathcal S \backslash \mathcal Q_0| > 0$. This implies, recalling \eqref{eq:tabxy}, 
\begin{align}\label{eq:tabxyHW} 
t_{a, b}^{-}(x, y, H, W) & = \int_{[0, 1]^{|V(H)|-2}}t_{a, b}^{-}(\bm z_{-(a, b)}, x, y, H, W) \prod_{r \notin \{a, b\}} \dd z_r \nonumber \\ 
& \geq  \int_{\mathcal S \backslash \mathcal Q_0} t_{a, b}^{-}(\bm z_{-(a, b)}, x, y, H, W) \prod_{r \notin \{a, b\}} \dd z_r > 0 , 
\end{align} 
since  $t_{a, b}^{-}(\bm z_{-(a, b)}, x, y, H, W) > 0$ on $\mathcal S \backslash \mathcal Q_0$. Recall that $(x,y)\in I'\times J'$ was chosen arbitrarily; hence \eqref{eq:tabxyHW} is true for all $(x,y)\in I'\times J'$. Further observe that
\begin{align}
    I'\times J'\subseteq \left\{P\bigcap\left(I'\times J'\right)\right\}\bigcup\left\{\left(I\times J\right)\setminus\left(P\bigcap\left(I\times J\right)\right)\right\},
\end{align}
implying
\begin{align*}
    \left|P\bigcap\left(I'\times J'\right)\right|
    & \geq |I'||J'|-\left|\left(I\times J\right)\setminus\left(P\bigcap\left(I\times J\right)\right)\right| \\
    & \geq |I'||J'|-\varepsilon|I||J| \tag*{(by  \eqref{eq:I-I'} and \eqref{eq:J-J'})} \\
    &\geq \left(\left(1-\frac{\varepsilon}{\delta}\right)^2-\varepsilon\right)|I||J| \\ 
    & =\left(0.81-\varepsilon\right)|I||J|>0 .
\end{align*} 
Therefore, recalling \eqref{eq:lower-bdd-sigma}
\begin{align}
    \sigma_{H,W}^2 & \geq c_H\int_{P\bigcap\left(I'\times J'\right)} t_{a, b}^{-}(x, y, H, W)^2 W(x,y)(1-W(x,y)) \ddx \ddy \nonumber \\ 
    & > 0 ,
\end{align} 
since by \eqref{eq:tabxyHW} and the definition of the set $P$, $t_{a,
  b}^{-}(x, y, H, W)^2 W(x,y)(1-W(x,y)) > 0$ for all $(x,y)\in P
\bigcap\left(I'\times J'\right)$. This shows that if $\sigma_{H, W}^2 = 0$ then $W$ is random-free. \qed 

\medskip

We conclude this section with an example (which generalizes the construction in \cite[Figure 1]{hladky2019limit} for triangles to general cliques) illustrating that Theorem \ref{thm:bizero} does not hold if the bipartite assumption is dropped (as mentioned in Remark \ref{remark:biHW}). 

\begin{example} \label{ExH} Suppose $H=K_r$ is the $r$-clique, for $r \geq 3$. Partition $[0, 1]$ into $2r$ intervals of measure $\frac{1}{2 r}$ each. 
Denote the first $r$ sets by $I_1, I_2, \ldots,  I_{r}$ and the next   $r$ sets by $J_1, J_2, \ldots,  J_{r}$.  Consider the following graphon: 
 \begin{align}\label{WHex}
  W(x, y) = 
\left\{
\begin{array}{cl}
 1 & \text{ for }  (x, y) \in (I_a \times I_b)  \text{ such that }  1 \leq a \ne b \leq r , \\ 
 1 & \text{ for }  (x, y) \in (J_a \times J_b)  \text{ such that } 1 \leq a \ne b \leq r , \\ 
 \frac{1}{2} & \text{ for }  (x,y)\in (I_1 \times J_1) \cup (J_1 \times I_1) , \\ 
  0 & \text{ otherwise. } 
\end{array}
\right.
\end{align}
In other words, $W$ is obtained by taking 2 disjoint graphon representations of $K_r$ (which corresponds to the complete $r$-partite graphon) inside $[0, \frac{1}{2}]^2$ and $[\frac{1}{2}, 1]^2$, respectively, and connecting the edges between the sets  $I_1$ and $J_1$ with probability $\frac{1}{2}$. Note that $t(K_r, W) > 0$. Denote $R := (I_1 \times J_1) \cup (J_1 \times I_1)$. By \eqref{eq:sigma-sq-exp}, 
\begin{align}
 \sigma_{H,W}^2= \frac{c_{K_r}}{4}&\sum_{\substack{ 1 \leq a \ne b \leq r \\ 1 \leq a \ne b \leq r }} \int_{R} t_{a, b}^{-}(x, y, K_r, W) t_{c, d}^{-}(x, y, K_r, W) \ddx \ddy . 
\label{eq:Hexample}
\end{align}      
Next, fix $1 \leq a \ne b \leq r$. If $(x,y)\in R$, then,
using the notation \eqref{eq:tabzxy}, 
\begin{align}\label{ab1}
t_{a, b}^{-}(\bm z_{-(a, b)}, x, y, K_r, W) = 0,   
\end{align} 
for all $\bm z_{-(a, b)}\in[0,1]^{r-2}$. Hence, for every $(x,y)\in R$, we have
$t_{a, b}^{-}(x, y, H, W) =0$ by \eqref{eq:tabxy}.
Consequently, it follows from \eqref{eq:Hexample} that
$ \sigma_{H,W}^2 = 0$. (In fact, $i_1,\dots,i_r$ can form an $r$-clique in $G(n,W)$ only if
$U_{i_1},\dots,U_{i_r}$ all belong to either $\bigcup_a I_a$ or $\bigcup_a J_a$;
hence the value of $W$ on $I_1\times J_1$ does not matter for $X_n(K_r,W)$.)
Moreover, \eqref{ab1} also implies that $\overline{t}(x, K_r, W)$ is
constant a.e., that is, $W$ is $K_r$-regular.
\end{example}

\section{Proof of Theorem \ref{thm:condition}}
\label{sec:conditionpf}

In the proof we will consider many equations or other relations 
that hold a.e.\ in $\oi$ or $\oi^2$. For this we use the notation that, for example, 
$\Sref{eq:condition1}$ denotes the set of all $(x,y)\in\oi^2$ such that the
equation in \eqref{eq:condition1} holds, and  
$\bSref{eq:condition1}$ denotes 
$\set{x\in\oi:(x,y)\in\Sref{eq:condition1}\text{ for a.e.}\ y\in\oi}$.
We use this notation  only for sets $\Srefdummy$ with full measure in
$\oi^2$;
note that then, by a standard application of Fubini's theorem, 
$\bSrefdummy$ has full measure in $\oi$, that is,
$x\in\bSrefdummy$ for a.e.\ $x\in\oi$.
Similarly, for relations with a single variable, we let, for example,
$\bSref{cl1}$ be the set of $x\in\oi$ such that the inequality
in \eqref{cl1} holds.

We tacitly assume $x,y,z\in\oi$ throughout the proof.
However, for notational convenience, we may write integrals with limits
that might be outside $\oi$; $\int_a^b$ should always be interpreted as
$\int_{[a,b]\cap\oi}$.  

For all $x\in[0,1]$, define $W_{x}:[0,1]\rightarrow[0,1]$ as
\begin{align}\label{eq:defWx}
    W_{x}(y):=W(x,y).
\end{align}
We regard $W_x$ as an element of $L^{2}[0,1]$.
Note that this means, in particular, that $W_x=W_y$ means
$W(x,z)=W(y,z)$ for a.e.\ $z$.
Since $W(x,y)$ is measurable and bounded,
it is well known that the mapping $x\mapsto W_x$ is a measurable, and
(Bochner) integrable, map $\oi\to L^2\oi$, see
\cite[Lemma III.11.16(b)]{dunford1988linear}.
The Lebesgue differentiation theorem holds for Bochner
integrable Banach space value functions,
see \cite[\S5.V]{bochner1933integration};
hence, a.e.\ $x\in\oi$ is a Lebesgue point of $x\mapsto W_x$. 
We will use $\norm{\cdot}$ and $\innprod{\cdot,\cdot}$ for the norm and inner
product in $L^2\oi$.

We will denote $t:= t(C_4, W)$. Suppose (to obtain a contradiction) that $t>0$, $W\not\equiv1$, 
but that the limit in \eqref{eq:limit2} is degenerate,
that is, $\Spec^-(W_{C_4})=\emptyset$ and $\gss_{C_4.W}=0$.
Then \eqref{eq:condition1} and \eqref{eq:condition2} both hold by Lemma
\ref{LC4}, and $W$ is random-free by Theorem
\ref{thm:bizero}, that is,
\begin{align}\label{rf}
  W(x,y)\in\set{0,1},
\qquad \text{a.e.\ } x,y.
\end{align}
We now separate the proof of the theorem into a sequence of claims.

\begin{claim}\label{CL1}
  For a.e.\ $x\in\oi$ and $W_x$ as defined in \eqref{eq:defWx},
  \begin{align}\label{cl1}
    \norm{W_x}\le (3t)^{1/4}
.  \end{align}
\end{claim}

\begin{proof}
By \eqref{eq:Uxy} and \eqref{eq:condition1},
for a.e.\ $(x,y)$,
\begin{align}\label{cl1a}
  \innprod{W_x,W_y} = U_1(x,y)
\le (3t)^{1/2}.
\end{align}
In particular, if $x\in\bSref{cl1a}$, then for every $\gd>0$,
\begin{align}
\left\langle W_x,  \frac{1}{2\gd}\int_{x-\gd}^{x+\gd} W_y \dd y\right\rangle
=
  \frac{1}{2\gd}\int_{x-\gd}^{x+\gd}  \innprod{W_x,W_y} \dd y\le (3t)^{1/2}.
\end{align}
If, furthermore, $x$ is a Lebesgue point of $x\mapsto W_x$, then it follows by
letting $\gd\to0$ that $\norm{W_x}^2\le(3t)^{1/2}$.
\end{proof}

\begin{claim}\label{C:CSW}
For a.e.\ $(x,y)\in\oi^2$,
  \begin{align}\label{csw}
W(x,y)=0 \implies W_x=W_y \text{ in } L^2\oi
\text{ and\/ }
\norm{W_x}=\norm{W_y}=(3t)^{1/4}. 
  \end{align}
\end{claim}

\begin{proof}
By \eqref{eq:Uxy} and \eqref{eq:condition1},
if $(x,y)\in\Sref{eq:condition1}$ and $W(x,y)=0$, then
\begin{align}\label{eq:U1prod}
\innprod{W_x,W_y}= U_{1}(x,y)=(3t)^{1/2}.
\end{align} 
If, furthermore, $x,y\in\bSref{cl1}$, then the Cauchy--Schwarz inequality
yields
\begin{align}
  (3t)^{1/2}=\innprod{W_x,W_y}
\le \norm{W_x}\norm{W_y}\le (3t)^{1/2}.
\end{align}
Hence, we must have equalities, and thus 
$\norm{W_x}=\norm{W_y}=(3t)^{1/4}$; 
moreover, equality in the Cauchy--Schwarz
inequality implies $W_x=W_y$.
\end{proof}

\begin{claim}\label{CL<1}
  We have $(3t)^{1/2}<1$.
\end{claim}

\begin{proof}
  Let $Z:=\set{(x,y):W(x,y)=0}$ and $Z':=Z\cap\Sref{csw}$. 
By \eqref{rf} and the assumption that $W$
  is not a.e.\ 1, we have $|Z'|=|Z|>0$.
For $x\in\oi$, let $Z'_x:=\set{y:(x,y)\in Z'}$.
By  Fubini's theorem, $\int_0^1|Z'_x|\dd x=|Z'|>0$, and thus
there exists  $x$ such that $|Z'_x|>0$.
Fix one such $x$. 
Then there exists $y\in Z'_x$, and thus $(x,y)\in Z'=Z\cap\Sref{csw}$.
Consequently, \eqref{csw} applies and yields $\norm{W_x}=(3t)^{1/4}$.
Furthermore,
$W(x,y)=0$ for all $y\in Z'_x$, and thus
\begin{align}
(3t)^{1/2}= \norm{W_x}^2=\int_0^1 W(x,y)^2\dd y \le 1-|Z'_x|<1.
\end{align}
\end{proof}

\begin{claim}\label{CL=}
  For a.e.\ $x\in\oi$,
  \begin{align}\label{cl=}
    \norm{W_x}= (3t)^{1/4}<1
.  \end{align}
\end{claim}

\begin{proof}
  Suppose $x\in  \bSref{rf} \cap\bSref{cl1}$.
Then, using Claim \ref{CL<1},
\begin{align}
|\set{y:W(x,y)>0}|
=
|\set{y:W(x,y)=1}|
=
\int_0^1 W(x,y)^2\dd y = \norm{W_x}^2\le (3t)^{1/2}<1.  
\end{align}
If, furthermore, $x\in\bSref{csw}$, this implies that 
there exists $y$ such that $W(x,y)=0$ and $(x,y)\in\Sref{csw}$, and thus, in
particular, 
$ \norm{W_x}= (3t)^{1/4}$.
The result \eqref{cl=} follows by Claim \ref{CL<1}.
\end{proof}

\begin{claim}\label{CLU2}
  For a.e.\ $(x,y)$,
  \begin{align}\label{clu2}
    W(x,y)>0 \implies U_2(x,y)>0.
  \end{align}
\end{claim}  

\begin{proof}
Let
\begin{align}
  L_1:=\set{(x,y)\in\oi^2:y \text{ is a Lebegue point of }y\mapsto W(x,y)}.
\end{align}
Then $L_1$ is measurable, and since for any given $x$, we have $(x,y)\in
L_1$ for a.e.\ $y$, it follows by Fubini's theorem that $|L_1|=1$, that is,
a.e. $(x,y)\in L_1$.

Now, assume that $(x,y)\in L_1$, $(y,x)\in L_1$ and that $(x,y)$ is a
Lebesgue point of the set $\set{(s,t):W(s,t)>0}$.
(In particular, $W(x,y)>0$.)
Let $\gd>0$ and let $I:=(x-\gd,x+\gd)$ and $J:=(y-\gd,y+\gd)$.
Then, if $\gd$ is small enough, 
\begin{align}
  |\set{s\in J:W(x,s)=0}|& < 0.1|J|,
\\
  |\set{t\in I:W(t, y)=0}|& < 0.1|I|,
\\
|\set{(s,t)\in J\times I:W(s,t)=0}|& < 0.1|I|\times|J|,
\end{align}
Then $W(x,s)W(s,t)W(t,y)>0$ on a subset of $I\times J$ of positive measure,
and thus $U_2(x,y)>0$.
  \end{proof}

\begin{claim}\label{CLW}
  For a.e.\ $(x,y)$,
\begin{align}\label{eq:Wxy}
	W(x,y)=1-\one\set{W_{x}=W_{y}}. 
\end{align} 
\end{claim}  

\begin{proof}
  Suppose $(x,y)\in\Sref{clu2}\cap\Sref{eq:condition1}$, and that
  $W(x,y)=1$.
Then $U_2(x,y)>0$ by \eqref{clu2}, and thus \eqref{eq:condition1} yields 
\begin{align}
  \innprod{W_x,W_y}=U_1(x,y) < (3t)^{1/2}.
\end{align}
If, furthermore, $x\in\bSref{cl=}$, it follows that $W_x\neq W_y$.

On the other hand,
if $(x,y)\in\Sref{csw}$ and $W(x,y)=0$, then $W_x=W_y$ by \eqref{csw}.

In both cases, \eqref{eq:Wxy} holds, and thus, using \eqref{rf},
\eqref{eq:Wxy} holds a.e.
\end{proof}

Since $W_{x}=W_{y}$  is an equivalence relation, 
 there exists a partition (possibly infinite) of
$\oi=\bigsqcup_{\alpha}B_{\alpha}$ such that 
if we define $\ga(x)$ for $x\in\oi$ by $x\in B_{\ga(x)}$, then
$W_x=W_y\iff\ga(x)=\ga(y)$, for all $x,y\in\oi$.
Note that each $B_{\ga}$ is measurable, since $x\mapsto W_x$ is.
We can write \eqref{eq:Wxy} as
\begin{align}\label{eq:Wf}
	W(x,y)=\one\set{\ga(x)\neq \ga(y)},
\qquad\text{for a.e. } (x,y)
.\end{align}

\begin{claim}\label{CLB}
  For a.e.\ $x\in\oi$,
  \begin{align}
    |B_{\ga(x)}|=1-(3t)^{1/2}.
  \end{align}
\end{claim}

\begin{proof}
  Suppose that $x\in\bSref{eq:Wxy} \cap \bSref{rf}\cap\bSref{cl=}$.
Then, 
\begin{align}
  |B_{\ga(x)}|  
&= \int_0^1\one\set{y\in B_{\ga(x)}}\dd y
= \int_0^1\one\set{W_y=W_x}\dd y
 = \int_0^1\bigpar{1-W(x,y)}\dd y\notag\\
&= 1-\int_0^1 W(x,y)\dd y 
= 1-\int_0^1 W(x,y)^2\dd y
= 1-(3t)^{1/2}.
\end{align}
\end{proof}

Since $1-(3t)^{1/2}>0$ by Claim \ref{CL=}, there can only be a finite number of
parts $B_\ga$ of measure $1-(3t)^{1/2}$, and by Claim \ref{CLB}, they fill up
$\oi$ except for a null set.
Hence, Claim \ref{CLB} and \eqref{eq:Wf}
imply that $W$ is a.e.\ equal to a 
complete multipartite graphon with equal part sizes (and thus  finitely many
parts).
In other words, after a measure preserving transformation, 
$W$ equals a.e.\ 
the graphon $W_K$ defined as follows, see Figure \ref{fig:figureW}.
Given  an integer $K \geq 1$, 
partition the interval $[0, 1]$ into $K$ intervals $I_1, I_2, \ldots, I_K$
of equal length ${1}/{K}$, and
define 
\begin{align}\label{eq:WKdef}
W_{K}(x, y) := 
  \begin{cases}
0  &   \text{ if }  (x, y) \in \bigcup_{s=1}^K I_s \times I_s , \\
1  &   \text {otherwise.}      
  \end{cases}
\end{align}

\begin{figure}[!h]
    \centering
    \includegraphics[width = 0.45\textwidth]{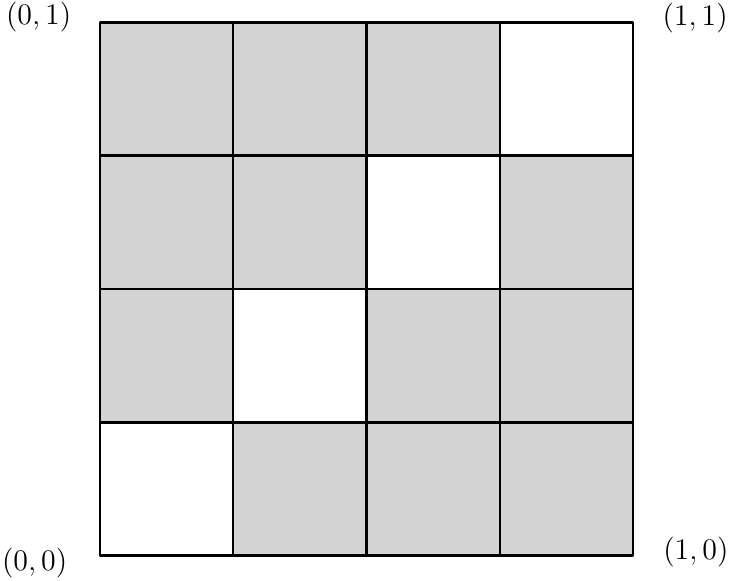}
    \caption{The graphon $W_{K}$ with $K=4$.}
    \label{fig:figureW}
\end{figure}

\begin{claim}\label{CLK}
  Let $W$ be the complete multipartite graphon $W_K$ with $K\ge2$ parts of equal
  sizes $1/K$. 
Then \eqref{eq:condition1} cannot hold.
\end{claim}

\begin{proof}
Suppose $W_K$ satisfies \eqref{eq:condition1} a.e. 
Then by Claim \ref{CLB}, each part must have size $1-(3t)^{1/2}$, that is,
$1-(3t)^{1/2} = 1/K$, which yields 
\begin{align}
	t(C_{4},W_K)=\frac{(K-1)^2}{3K^2} . 
	\end{align} 
On the other hand, a direct calculation shows that 	
	\begin{align}
		t(C_{4},W_K)=\frac{(K-1)^4+(K-1)}{K^4}. 
	\end{align}
We thus must have $\frac{(K-1)^2}{3K^2} = \frac{(K-1)^4+(K-1)}{K^4}$,
which simplifies to
\begin{align}\label{quad}
  K(K-1)(2K^2-8K+9)=0,
\end{align}
which is impossible. (The only real roots to \eqref{quad} are $K=0$ and $K=1$.)
\end{proof}

Claim \ref{CLK} gives the desired contradiction and 
completes the proof of Theorem \ref{thm:condition}.

\section{Proof of Theorem \ref{thm:K12Z}}
\label{sec:K12Zpf}

The proof is similar to that of Theorem \ref{thm:condition}. Here we will
denote $t:= t(K_{1, 2}, W)=\int d_W(x)^2\dd x$. 
Suppose that $t>0$, $W\not\equiv1$, but that 
$\Spec^-(W_{K_{1,2}})=\emptyset$ and $\gss_{K_{1,2}.W}=0$.
Then \eqref{eq:K12-cond1} and \eqref{eq:K12-cond2} both hold by Lemma
\ref{lem:LK12}, and $W$ is random-free by Theorem
\ref{thm:bizero}, that is,
$W(x,y)\in \{0,1\}$ for  a.e.\ $x,y\in [0,1]^2$. 
Now, recalling the definition of $W_{x}$
from \eqref{eq:defWx} we have the following claim, which can be proved by
arguments similar to Claims \ref{CL1}, \ref{C:CSW}, \ref{CL<1}, and
\ref{CL=}.

\begin{claim}\label{claim:K121}
For a.e.\ $(x,y)\in [0,1]^2$,
\begin{align}
	W(x,y)= 0 \implies W_{x}=W_{y}\text{ in }L^{2}[0,1]\text{ and }\|W_{x}\|_{2}=\|W\|_{y}=(3t)^{1/2}.
\end{align}
Moreover,  for a.e.\ $x\in [0,1]$, $\|W_{x}\|_{2}=(3t)^{1/2}<1$. 
\end{claim}


Next, we have the analogue of Claim \ref{CLU2} for the 2-star. 

\begin{claim}
For a.e.\ $(x,y)\in [0,1]^2$,
\begin{align}
	W(x,y)>0\implies d_{W}(x)+d_{W}(y)>0 . 
\end{align}
\end{claim}
\begin{proof}
	Similarly to the proof of Claim \ref{CLU2}, for a.e.\ $(x,y)\in[0,1]^2$
such that $W(x,y)>0$, we can choose $\delta>0$ small enough
    such that for 
$J=(y-\delta,y+\delta)$, 
	\begin{align}
\left|\left\{s\in J:W(x,s)=0\right\}\right|<0.1|J| 
.	\end{align}
This implies that the set 
$\set{s\in \oi:W(x,s)>0}$
has positive measure, 
and thus
$d_W(x)>0$.
\end{proof}

Now, as in Claim \ref{CLW}, it follows that for a.e.\ $(x,y)\in [0,1]^2$, 
\begin{align}\label{vaf}
W(x,y) = 1-\bm{1}\left\{W_{x}=W_{y}\right\}.  
\end{align}
As in the proof of Theorem \ref{thm:condition},
the equivalence relation $W_{x}=W_{y}$ defines  a possibly infinite
partition of $[0,1]=\bigsqcup_{\alpha}B_{\alpha}$. 
For $x\in [0,1]$ define $\alpha(x)$ to be the index such that $x\in
B_{\alpha(x)}$. 
Then, by definition,  $W_{x}=W_{y}\iff \alpha(x)=\alpha(y)$, 
which by \eqref{vaf} yields,
for a.e.\ $x\in [0,1]$,
\begin{align}\label{eq:Wform}
	W(x,y) = \bm{1}\left\{\alpha(x)\neq\alpha(y)\right\} . 
\end{align}
Again, similarly to Claim \ref{CLB} we have for a.e.\ $x\in [0,1]$,
\begin{align}\label{eq:bsize}
	\left|B_{\alpha(x)}\right| = 1-3t . 
\end{align}
Note that by Claim \ref{claim:K121}, $1-3t>0$. Hence, 
by \eqref{eq:bsize},
there can only be a
finite number of parts $B_{\alpha}$ of positive measure and 
the remaining parts have together measure $0$. 
Therefore, by \eqref{eq:Wform} and \eqref{eq:bsize}  we conclude that after
a measure preserving transformation, $W$ must be of the form $W_{K}$ as
defined in \eqref{eq:WKdef} for some $K\geq 1$. 
We have excluded $W\equiv1$, so $K>1$. 

\begin{claim}\label{claim:Kimpossible}
Let $W=W_{K}$ for some $K\geq 2$. Then  \eqref{eq:K12-cond1} cannot hold.
\end{claim}
\begin{proof} Suppose $W_K$ satisfies \eqref{eq:condition1} a.e. Then by \eqref{eq:bsize}, each part 
must have size $1-3t$, that is, $1-3t = 1/K$.
In other words, 
	\begin{align}
		t(K_{1, 2}, W_K) = \frac{K-1}{3K} . 
	\end{align}
	On the other hand, since $d_{W_K}(x) = \frac{K-1}{K}$ a.e., 
	\begin{align}
		t(K_{1, 2}, W_K) = \int_0^1 d_{W_K}(x)^2 \dd x = \frac{(K-1)^2}{K^2} . 
	\end{align}
	Thus we must have $\frac{K-1}{3K} = \frac{(K-1)^2}{K^2}$, that is, $K = \frac{3}{2}$, which is impossible. 
\end{proof}

Claim \ref{claim:Kimpossible} gives a contradiction and
completes the proof of Theorem \ref{thm:K12Z}.

\small 
\bibliographystyle{abbrvnat}
\bibliography{graphonbib}

\end{document}